\documentclass{amsart}
\usepackage{amsmath}
\usepackage{amsfonts}
\usepackage{amssymb}
\usepackage{amsthm}
\usepackage{enumitem}

\newtheorem{df}{Definition}

\newtheorem{theorem}{Theorem}
\newtheorem{proposition}{Proposition}
\newtheorem{lemma}{Lemma}
\newtheorem*{rem}{Remark}

\begin{document}
\title[Nonexistence of $D(4)$-quintuples]{NONEXISTENCE OF $D(4)$-QUINTUPLES}
\author{MARIJA BLIZNAC TREBJE\v{S}ANIN, ALAN FILIPIN}

\begin{abstract} In this paper we prove a conjecture that $D(4)$-quintuple does not exist using both classical and new methods.  Also, we give a new version of the Rickert's theorem that can be applied on some $D(4)$-quadruples.
\end{abstract}

\maketitle


\section{Introduction}

\begin{df} Let $n\neq0$ be an integer. We call the set of $m$ distinct positive integers a $D(n)$-$m$-tuple, if the product of any two of its distinct elements increased by $n$ is a perfect square.
\end{df}

One of the most interesting and most studied questions is how large those sets can be. In this paper, we will consider only $D(4)$-quintuples $\{a,b,c,d,e\}$, such that $a<b<c<d<e$. It is conjectured in \cite{dujram} that all $D(4)$-quadruples, such that $a<b<c<d$, are regular, i.e.
$$d=d_+=a+b+c+\frac{1}{2}(abc+\sqrt{(ab+4)(ac+4)(bc+4)}),$$
which implies that there does not exist a $D(4)$-quintuple.\\

The second author in \cite{ireg_pro} has proven that an irregular $D(4)$-quadruple cannot be extended to a quintuple with a larger element and in \cite{konacno_petorki} that there are at most $4$ ways to extend a $D(4)$-quadruple to a quintuple with a larger element. The best published upper bound on the number of $D(4)$-quintuples is $6.8587 \cdot 10^{29}$ found by the authors in \cite{nas}. \par
Case $n=1$ is the most famous and mostly studied. Dujella proved in \cite{duje_kon} that a $D(1)$-sextuple does not exist and that there are at most finitely many quintuples. Over the years many authors improved the upper bound for the number of $D(1)$-quintuples and finally, very recently, He, Togb\'{e} and Ziegler in \cite{petorke} announced the proof of the nonexistence of $D(1)$-quintuples. To see all details of the history of the problem with all references one can visit the webpage \cite{duje_web}. \par
Our approach was to use the methods and approach from \cite{petorke} and apply them to $D(4)$-quintuples, but modifications were necessary since not all previously proven results are comparable in the cases $n=1$ and $n=4$. One of the main differences is that the result from \cite[Theorem A.]{cipu_glasnik}, where authors proved that $b>3a$ in $D(1)$-quintuple, cannot be proven for $D(4)$ case using the exactly same methods. But, in $D(4)$ case we have $b\geq a+57\sqrt{a}$, proven by the second author in \cite{fil_par}, which can be used with some modifications to prove similar auxiliary results as in \cite{petorke}. Throughout the paper we will give a proof only for the statements which differ from the $D(1)$ case, where the modification of the proof or some new idea was necessary, or some additional explanation is needed because not all of the proofs from \cite{petorke} have been clearly explained or there were some gaps in the version we are referring to. Thus, we did not take all results from \cite{petorke} for granted. \par
One of the sections of the paper will be dedicated to using methods from \cite{cff} to get an improved version of Rickert's theorem for $D(4)$-quadruples and use it to get the bounds on elements of a $D(4)$-quintuple in the last section of the paper which was necessary to prove our result. \par
The last two sections will be dedicated to proving the main result of our paper. Our main result is the following theorem.

\begin{theorem}\label{glavni}
There does not exist a $D(4)$-quintuple.
\end{theorem}

Let us mention that stronger version of conjecture, i.e. that all quadruples are regular, still remains open.

\section{Known results about elements of a $D(4)$-$m$-tuple}

For a $D(4)$-triple $\{a,b,c\}$, $a<b<c$, we define
$$d_{\pm}=d_{\pm}(a,b,c)=a+b+c+\frac{1}{2}(abc\pm \sqrt{(ab+4)(ac+4)(bc+4)}),$$
and it is easy to check that $\{a,b,c,d_{+}\}$  is a $D(4)$-quadruple, which we will call regular quadruple, and if $d_{-}\neq 0$ then $\{a,b,c,d_{-}\}$ is also a regular $D(4)$-quadruple with $d_{-}<c$.
Also we will use standard notation $r=\sqrt{ab+4}$, $s=\sqrt{ac+4}$ and $t=\sqrt{bc+4}$.
\begin{lemma}\label{c_granice}
Let $\{a,b,c\}$ be a $D(4)$-triple and $a<b<c$. Then $c=a+b+2r$ or $c>\max\{ab,4b\}$.
\end{lemma}
\begin{proof}
 This follows from \cite[Lemma 3]{fil_xy4} and \cite[Lemma 1]{duj}.
\end{proof}

The next lemma can be proven similarly as \cite[Lemma 2]{petorke}.

\begin{lemma}\label{granice_za_d_plus}
Let $\{a,b,c\}$ be a $D(4)$-triple and $a<b<c$. Then $abc+c<d_+<abc+4c$.
\end{lemma}

Results from the next two lemmas will be used in the rest of the paper very often, so sometimes we will not reference them.

\begin{lemma}\cite[Lemmas 2.2 and 2.3]{nas}\label{nas_rez_granice}
Let $\{a,b,c,d,e\}$ be a $D(4)$-quintuple such that $a<b<c<d<e$. Then $b>10^5$. Also, if $c\neq a+b+2r$, then $b>4a$.
\end{lemma}

\begin{lemma}\cite[Corollary 1.2]{fil_par}\label{b_a_57sqrt}
If  $\{a,b,c,d,e\}$ is a $D(4)$-quintuple such that $a<b<c<d<e$, then $b\geq a+57\sqrt{a}$.
\end{lemma}

From \cite{ireg_pro} we also have that an element $d$ in a $D(4)$-quintuple $\{a,b,c,d,e\}$ is uniquely determined by the triple $\{a,b,c\}$.

\begin{lemma}
If  $\{a,b,c,d,e\}$ is a $D(4)$-quintuple such that $a<b<c<d<e$, then $d=d_+$.
\end{lemma}

\section{New version of Rickert's theorem}\label{Rickert_section}

In this section we will prove a new version of Rickert's theorem similar to the one in \cite{cff}, which is essential to finding some upper bounds on the elements of $D(4)$-quintuple when $c>a+b+2r$. Unfortunately, in the $D(4)$ case we could not get all results analogously as in \cite{cff} for a $D(1)$-quintuple, but still, these results will be essential for proving our main result. \par

All the results in this section and its proofs are analogous to the ones from \cite{cff} so we will give them without a proof.

\begin{theorem}\label{teorem2.1}
Put $A'=\max\{4(B-A),4A\}$ and $g=\gcd (A,B)$ and let $A$, $B$  be integers with $0<A/g \leq B/g-4$, $B/g\geq 5$ and $N$ a multiple of $AB$. Assume that  $N\geq 59.488 A'B^2(B-A)^2g^{-4}$.  Then the numbers $$\theta_1=\sqrt{1+\frac{4B}{N}}\quad \text{  and  }\quad\theta_2=\sqrt{1+\frac{4A}{N}}$$
satisfy
$$\max\left\{ \left|\theta_1-\frac{p_1}{q}\right|,\left| \theta_2 - \frac{p_2}{q} \right| \right\}>\left(\frac{3.53081\cdot 10^{27}A'BN}{Ag^2}\right)^{-1}q^{-\lambda}$$
for all integers $p_1$, $p_2$, $q>0$, where
$$\lambda=1+\frac{\log(2.500788A^{-1}A'BNg^{-2})}{\log(0.04216N^2g^2A^{-1}B^{-1}(B-A)^{-2})}<2.$$

\end{theorem}

\medskip
Let $\{A,B,C\}$ be a $D(4)$-triple which can be extended to a quadruple with an element $D$. Then there exist positive integers $x,y,z$ such that
$$AD+4=x^2,\quad BD+4=y^2, \quad CD+4=z^2.$$
By expressing $D$ from these equations we get the following system of generalized Pell equations
\begin{align*}
Cx^2-Az^2&=4(C-A),\\
Cy^2-Bz^2&=4(C-B).
\end{align*}

Solutions of each of these equations can be expressed with a binary recurrent sequences as described in details in \cite{sestorka}, and we will denote them  $z=v_m=w_n$, where $m$ and $n$ are some positive integers. If this  quadruple is contained in a $D(4)$-quintuple, then from \cite{konacno_petorki} we know that $m$ and $n$ are even and  we will consider only that case. \par

\begin{lemma}\label{lema_donjaogradaza_logz}
Suppose that there exist positive integers $m$ and $n$ such that $z=v_{2m}=w_{2n}$, and $|z_1|=2$, and that $C\geq B^2\geq 25$. Then $\log z>n \log BC$.
\end{lemma}

And finally we get a new version of the Rickert's theorem.

\begin{lemma}\label{ricket_konacno}
Suppose that there exist integers $m\geq 3$ and $n\geq 2$ such that $z=v_{2m}=w_{2n}$ and $|z_1|=2$ and that for $A'=\max\{4(B-A),4A\}$ and $g=\gcd (A,B)$, $A$ and $B$ are integers such that $0<A/g \leq B/g-4$, $B/g\geq 5$ and $N$ is a multiple of $AB$, such that  $N\geq 59.488 A'B^2(B-A)^2g^{-4}$. Then
\begin{equation*}n<\frac{4\log(8.40335\cdot 10^{13}(A')^{\frac{1}{2}}A^{\frac{1}{2}}B^2Cg^{-1})\log(0.20533A^{\frac{1}{2}}B^{\frac{1}{2}}C(B-A)^ {-1}g)}{\log(BC)\log(0.016858A(A')^{-1}B^{-1}(B-A)^{-2}Cg^4)}.
\end{equation*}
\end{lemma}

Now we will use these results to prove an upper bound on the element $c$ in a $D(4)$-quintuple in the terms of smaller elements $a$ and $b$.

\begin{proposition}\label{granicagornja_na_c}
Let $\{a,b,c,d,e\}$ be a $D(4)$-quintuple such that $a<b<c<d<e$. Then
$$c<\frac{237.952b^3}{a}.$$
\end{proposition}
\begin{proof}
If $c=a+b+2r$, then $c<4b<\frac{237.952b^3}{a}$.\\
Let us now assume that $c\neq a+b+2r$ and that $d\geq k\cdot b^4$ for some positive real number $k$. From Lemmas  \ref{c_granice} and \ref{nas_rez_granice} we know that $b>10^5$, $c>\max\{ab,4b\}$ and $b>4a$. Then
$$
\frac{59.488A'B(B-A)^2}{Ag^4}<237.952(b-a)^3\cdot \frac{b}{a}<237.952b^4,
$$
which implies that we can use Lemma \ref{ricket_konacno}\@ for $k=237.952$. Now we observe
\begin{align*}
8.40335\cdot 10^{13}(A')^{\frac{1}{2}}A^{\frac{1}{2}}B^2Cg^{-1}
&<8.40335\cdot 10^{13} b^3d,\\
0.20533A^{\frac{1}{2}}B^{\frac{1}{2}}C(B-A)^ {-1}g
&<0.03423bd,\\
0.016858A(A')^{-1}B^{-1}(B-A)^{-2}Cg^4
&>0.0042145b^{-4}d,
\end{align*}
and get
$$n<\frac{4\log(8.40335\cdot 10^{13} b^3d)\log(0.03423bd)}{\log(bd)\log(0.0042145b^{-4}d)}.$$
It can be shown that the right hand side is decreasing in $d$ and since $d\geq 237.952b^4$, we can now observe 
$$n<\frac{4\log(1.9996\cdot 10^{16} b^7)\log(8.14272b^5)}{\log(237.952b^5)\log(1.002848)}.$$
From the proof of \cite[Lema 3.2.]{nas} we know that in a $D(4)$-quadruple it holds $m\geq 0.618034\sqrt{d/b}$, so
$$n>0.309017\sqrt{\frac{d}{b}}>0.309017\sqrt{237.952}b^{3/2}>4.7668b^{3/2}.$$
By combining the inequalities, we get $b<803$, which cannot be true. So we have $d<237.952b^4$ which implies $abc<237.952b^4$, i.e.
$$c<\frac{237.952b^3}{a}.$$
\end{proof}

\section{An operator on Diophantine triples}

An operator on triples, defined for the first time by He, Togb\'{e} and Ziegler in \cite{petorke}, has been shown to be one of the crucial steps in proving the nonexistence of $D(1)$-quintuples. The same will be true for the $D(4)$ case, so here we define it similarly and state some analogous results concerning the operator on $D(4)$-triples. However, we slightly extend their definition.

\begin{df}
A $D(4)$-triple $\{a,b,c\}$, $a<b<c$, is called an Euler or a regular triple if $c=a+b+2r$.
\end{df}

For a regular triple $\{a,b,c\}$ it is easy to prove that $d_{+}(a,b,c)=rst$ and $s=a+r$, $t=b+r$.\\
The following statements about regular triples will be given without proof, since they are easy to prove as in $D(1)$ case.

\begin{proposition}\label{dje0}
The $D(4)$-triple $\{a,b,c\}$ is a regular triple if and only if \\$d_{-}(a,b,c)=0$.
\end{proposition}

\begin{proposition}\label{drugaprop}
Let $\{a,b,c\}$ be a $D(4)$-triple, such that $a<b<c$. We have
$$a=d_{-}(b,c,d_{+}(a,b,c)),\quad b=d_{-}(a,c,d_{+}(a,b,c)),\quad c=d_{-}(a,b,d_{+}(a,b,c)).$$
Moreover, if $\{a,b,c\}$ is not a regular triple, then
$$c=d_{+}(a,b,d_{-}(a,b,c)).$$
In particular $\{a,b,d_{-}(a,b,c),c\}$ is a regular $D(4)$-quadruple.
\end{proposition}

Now we will define an operator on $D(4)$-triples. The idea follows from the fact that any $D(4)$-triple can be extended with a larger element to a $D(4)$-quadruple $\{a,b,c,d_{+}\}$. Hence, we obtain three new $D(4)$-triples, $\{a,b,d_{+}\}$, $\{a,c,d_{+}\}$ and $\{b,c,d_{+}\}$ which we may consider to be farther away from a regular triple than the original triple $\{a,b,c\}$. We can reverse this observation and define the following operator.

\begin{df}
We define $\partial$ to be an operator which sends a non-regular $D(4)$-triple $\{a,b,c\}$ to a $D(4)$-triple $\{a',b',c'\}$ such that
$$\partial(\{a,b,c\})=\{a,b,c,d_{-}(a,b,c)\}\setminus \{\max(a,b,c)\}.$$ If $D(4)$-triple $\{a,b,c\}$ is a regular triple, then we define that $\partial$ sends this triple to the same $D(4)$-triple $\{a,b,c\}$, i.e.
$$\partial(\{a,b,c\})=\{a,b,c\}.$$

For $D\in \mathbb{N}_{0}$ we can define the operator $\partial_{-D}$ on the set of $D(4)$-triples recursively as follows.
\begin{enumerate}
\item For any $D(4)$-triple $\{a,b,c\}$ we define
$$\partial_{0}(\{a,b,c\})=\{a,b,c\}.$$
\item We recursively define
$$\partial_{-D}(\{a,b,c\})=\partial(\partial_{-(D-1)}(\{a,b,c\})),\quad \textit{ for }D\geq 1.$$
\end{enumerate}
Moreover, we put
$$d_{-D}(a,b,c)=d_{-}(\partial_{-(D-1)}(\{a,b,c\})).$$
In particular, $\partial=\partial_{-1}$ and
$\partial_{-2}(\{a,b,c\})=\partial(\partial_{-1}(\{a,b,c\})).$
\end{df}

\begin{rem}
Observe that by using operator $\partial$ repeatedly, for a fixed triple $\{a,b,c\}$ we get an infinite sequence of $D(4)$-triples $$\partial_{0}(\{a,b,c\}), \partial_{-1}(\{a,b,c\}), \partial_{-2}(\{a,b,c\}),\dots,\partial_{-D}(\{a,b,c\}),\dots.$$ In the next Proposition we will show that for each $D(4)$-triple this sequence becomes stationary after $D$-th element for some $D$, which implies that every triple can be obtained from a regular triple using extensions with $d_{+}$ element explained before. Also, we will show that the repeating element is a regular triple, and give an upper bound for the number $D$.
\end{rem}

\begin{proposition}
For any fixed $D(4)$-triple $\{a,b,c\}$ there exists a minimal nonnegative integer $D<\frac{\log(abc)}{\log 5}$ such that $d_{-(D+1)}(a,b,c)=0.$
\end{proposition}
\begin{proof}
For a regular triple $\{a,b,c\}$ we have for each $D\in \mathbb{N}_0$ that $d_{-(D+1)}(a,b,c)=0$ since $\partial_{-D}\{a,b,c\}=\{a,b,c\}$, so minimal $D$ is $D=0$. For a non-regular triple, the idea is to use the fact that $c>abd_{-1}(a,b,c)$ and $a'b'c'=abd_{-1}(a,b,c)<\frac{abc}{5}$ since $ab\geq 5$. We can see that by using the operator $\partial$ for $k$ times we get  $a'b'c'<\frac{abc}{5^k}$, so we must get $d_{-1}(a',b',c')=0$ for some $\{a',b',c'\}$ and the result follows from Proposition \ref{dje0}.
\end{proof}

\begin{df}
For a $D(4)$-triple $\{a,b,c\}$ we will say that it has a degree $D$ and that it is generated by a regular triple $\{a',b',c'\}$ if $D$ is minimal such that $d_{-(D+1)}(a,b,c)=0$ and $\partial_{-D}(\{a,b,c\})=\{a',b',c'\}$. If the triple $\{a,b,c\}$ is of degree $D$ we will write $\deg(a,b,c)=D.$
\end{df}

\begin{rem}
Let us now observe an example of these definitions. The $D(4)$-triple $\{1,5,12\}$ generates $3$ triples, $\{1,5,96\}$,  $\{1,12,96\}$ and $\{5,12,96\}$,  of degree $1$, and $9$ triples of degree $2$, one of them is, for example, $\{1,12,1365\}$. It is clear that by induction, each $D(4)$-triple generates $3^k$ triples of degree $k$.
\end{rem}

\section{System of Pell equations}

Let $\{a,b,c\}$ be a $D(4)$-triple, $a<b<c$, and $r,s,t$ positive integers such that
$$ab+4=r^2,\quad ac+4=s^2, \quad bc+4=t^2.$$
Suppose that $\{a,b,c,d,e\}$ is a $D(4)$-quintuple, $a<b<c<d<e$, and as before
$$ad+4=x^2,\quad bd+4=y^2, \quad cd+4=z^2,$$
 $x,y,z\in \mathbb{N}$. Then, there also exist integers $X,Y,Z,W$ such that
$$ae+4=X^2,\quad be+4=Y^2,\quad ce+4=Z^2,\quad de+4=W^2.$$
From \cite[Theorem 1]{ireg_pro} we have $d=d_+$, which implies
$$x=\frac{at+rs}{2},\quad y=\frac{bs+rt}{2},\quad z=\frac{cr+st}{2}.$$
By eliminating $e$ from the equations above, we get a system of generalized Pell equations
\begin{align}
aY^2-bX^2&=4(a-b), \label{pelljdbaAB}\\
aZ^2-cX^2&=4(a-c), \label{pelljdbaAC}\\
bZ^2-cY^2&=4(b-c), \label{pelljdbaBC}\\
aW^2-dX^2&=4(a-d), \label{pelljdbaAD}\\
bW^2-dY^2&=4(b-d), \label{pelljdbaBD}\\
cW^2-dZ^2&=4(c-d). \label{pelljdbaCD}
\end{align}
The next lemma, which is a part of Lemma 2\@ in \cite{dujram}, gives us a description of solutions of Pell equations (\ref{pelljdbaAB})-(\ref{pelljdbaCD}).
\begin{lemma}\label{opcerjpellove}
If $(X,Y)$ is a positive integer solution to a generalized Pell equation
$$aY^2-bX^2=4(a-b),$$
with $ab+4=r^2$, then it is obtained from
$$Y\sqrt{a}+X\sqrt{b}=(y_0\sqrt{a}+x_0\sqrt{b})\left(\frac{r+\sqrt{ab}}{2}\right)^n,$$
where $n\geq0$ is an integer and $(x_0,y_0)$ is integer solution of the equation such that
$$1\leq x_0 \leq \sqrt{\frac{a(b-a)}{r-2}},\quad \textit{and}\quad 1\leq |y_0| \leq \sqrt{\frac{(r-2)(b-a)}{a}}.$$
\end{lemma}

By applying this Lemma to the equations (\ref{pelljdbaAB})-(\ref{pelljdbaCD}) we obtain
\begin{align}
Y\sqrt{a}+X\sqrt{b}&=Y_{h'}^{(a,b)}\sqrt{a}+X_{h'}^{(a,b)}\sqrt{b}=(Y_0\sqrt{a}+X_0\sqrt{b})\left(\frac{r+\sqrt{ab}}{2}\right)^{h'} \label{rjjdbaAB}\\
Z\sqrt{a}+X\sqrt{c}&=Z_{j'}^{(a,c)}\sqrt{a}+X_{j'}^{(a,c)}\sqrt{c}=(Z_1\sqrt{a}+X_1\sqrt{c})\left(\frac{s+\sqrt{ac}}{2}\right)^{j'}  \label{rjjdbaAC}\\
Z\sqrt{b}+Y\sqrt{c}&=Z_{k'}^{(b,c)}\sqrt{b}+Y_{k'}^{(b,c)}\sqrt{c}=(Z_2\sqrt{b}+Y_2\sqrt{c})\left(\frac{t+\sqrt{bc}}{2}\right)^{k'}  \label{rjjdbaBC}\\
W\sqrt{a}+X\sqrt{d}&=W_{l'}^{(a,d)}\sqrt{a}+X_{l'}^{(a,d)}\sqrt{d}=(W_3\sqrt{a}+X_3\sqrt{d})\left(\frac{x+\sqrt{ad}}{2}\right)^{l'}  \label{rjjdbaAD}\\
W\sqrt{b}+Y\sqrt{d}&=W_{m'}^{(b,d)}\sqrt{b}+Y_{m'}^{(b,d)}\sqrt{d}=(W_4\sqrt{b}+Y_4\sqrt{d})\left(\frac{y+\sqrt{bd}}{2}\right)^{m'} \label{rjjdbaBD}\\
W\sqrt{c}+Z\sqrt{d}&=W_{n'}^{(c,d)}\sqrt{c}+Z_{n'}^{(c,d)}\sqrt{d}=(W_5\sqrt{c}+Z_5\sqrt{d})\left(\frac{z+\sqrt{cd}}{2}\right)^{n'} \label{rjjdbaCD}
\end{align}
where $h',j',k',l',m',n'$ are nonnegative integers, and $Y_0$, $Y_2$, $Y_4$, $X_0$, $X_1$, $X_3$, $Z_1$, $Z_2$, $Z_5$, $W_3$, $W_4$, $W_5$ integers which satisfy appropriate inequalities from Lemma \ref{opcerjpellove}.\\
Each sequence of solutions can be expressed as a pair of binary recurrence sequences, so for example, a sequence of solutions $(Y_{h'}^{(a,b)},X_{h'}^{(a,b)})$ to equation (\ref{rjjdbaAB}) satisfy the following recursions:
\begin{align*}
Y_{0}^{(a,b)}=Y_0,&\quad Y_{1}^{(a,b)}=\frac{rY_0+bX_0}{2},\quad Y_{h'+2}^{(a,b)}=rY_{h'+1}^{(a,b)}-Y_{h'}^{(a,b)},\\
X_{0}^{(a,b)}=X_0,&\quad X_{1}^{(a,b)}=\frac{rX_0+aY_0}{2},\quad X_{h'+2}^{(a,b)}=rX_{h'+1}^{(a,b)}-X_{h'}^{(a,b)},
\end{align*}
which can easily be proven by induction. \\
We will now state and prove some lemmas about initial values of the sequences of solutions and about its indices $h',j',l',k',m',n'$.
\begin{lemma}{\cite[Lemma 3]{konacno_petorki}}\label{Wjepm2}
If $W=W_{l'}^{(a,d)}=W_{m'}^{(b,d)}=W_{n'}^{(c,d)}$, then we have $l'\equiv m'\equiv n' \equiv 0 \ (\bmod \ 2)$. Also, $$W_3=W_4=W_5=2 \varepsilon=\pm 2\quad \textrm{and} \quad X_3=Y_4=Z_5=2.$$
\end{lemma}
 In the next lemma we will prove a similar result about remaining indices and initial values of sequences. Proof defers from the one in \cite{petorke} so we give it in detail.
\begin{lemma}
We have $h'\equiv j'\equiv k' \equiv 0 \ (\bmod \ 2) $ and
$$X_0=X_1=Y_0=Y_2=Z_1=Z_2=2.$$
\end{lemma}
\begin{proof}
Let us consider the system of the equations (\ref{pelljdbaAB}) and (\ref{pelljdbaBD}). \\
From Lemma \ref{opcerjpellove} we have the bound on $Y_0$,
$|Y_0|<b^{3/4}a^{-1/4}$,
and since $Y_{h'}^{(a,b)}$ satisfy recursion
$$Y_{0}^{(a,b)}=Y_0,\quad Y_{1}^{(a,b)}=\frac{rY_0+bX_0}{2},\quad Y_{h'+2}^{(a,b)}=rY_{h'+1}^{(a,b)}-Y_{h'}^{(a,b)},$$
we easily see that
$$Y_{h'}^{(a,b)}\equiv
\begin{cases}
Y_0^{(a,b)}(\bmod \ b),&\ h' \text{ even} ,\\
Y_1^{(a,b)}(\bmod \ b),&\ h' \text{ odd}.
\end{cases}$$
On the other hand, for $Y_{m'}^{(b,d)}$, from Lemma \ref{Wjepm2}, we have
$$Y_{0}^{(b,d)}=Y_4=2,\quad Y_{1}^{(b,d)}=y+\varepsilon b,\quad Y_{m'+2}^{(b,d)}=yY_{m'+1}^{(b,d)}-Y_{m'}^{(b,d)}$$
and since we know that $m'$ is even, we obtain $Y_{m'}^{(b,d)}\equiv 2 (\bmod \ b) .$\\
We consider $Y_{h'}^{(a,b)}=Y_{m'}^{(b,d)}$ and let us assume that $h'$ is odd.
Then
$$
\frac{1}{2}(rY_0+bX_0)\equiv 2 (\bmod \ b)
$$
and since $bX_0\equiv 0 (\bmod \ b)$, after subtracting the first congruence equation from the second we have $\frac{1}{2}(bX_0-rY_0)\equiv -2 (\bmod \ b)$.
Now, we observe
\begin{align*}
\nonumber (bX_0-rY_0)(bX_0+rY_0)&=b^2X_0^2-r^2Y_0^2=b(aY_0^2+4(b-a))-abY_0^2-4Y_0^2\\
&=4b(b-a)-4Y_0^2.
\end{align*}
Since $|Y_0|<b^{3/4}a^{-1/4}$, we have
$$4b(b-a)-4Y_0^2>4b(b-a)-4b^{3/2}a^{-1/2}=4b(b-a-(b/a)^{1/2}),$$
and since the right hand side is increasing in $b$, and from Lemma \ref{b_a_57sqrt}\@ we know that $b \geq a+57\sqrt{a}$, we get
$$b-a-(b/a)^{1/2}>57\sqrt{a}-\left( 1+\frac{57}{\sqrt{a}}\right)^{1/2}>0.$$
So we can conclude that $bX_0-r|Y_0|>0$. On the other hand, we can easily see that $b^2X_0^2-r^2Y_0^2<4b^2$, i.e. $\frac{1}{2}(bX_0-r|Y_0|)<b$.\par
Now, let us consider separately these cases: \\
$1.)$ If $Y_0>0$, then $\frac{1}{2}(bX_0-r|Y_0|)=\frac{1}{2}(bX_0-rY_0)$ so we must have $\frac{1}{2}(bX_0-rY_0)=b-2$. Observe that
\begin{align*}
bX_0+rY_0&<\frac{4b^2}{bX_0-rY_0}=\frac{4b^2}{2b-4}\\
&=2b+\frac{8b}{2b-4}=2b+4+\frac{16}{2b-4}<2b+4.1.
\end{align*}
Since $b>10^5$ implies $r>316$ and both addends on the right hand side of the inequality are positive, the only options for $X_0$ are $X_0=1$ and  $X_0=2$.
If $X_0=1$, by direct computation we can see that there is no $Y_0$ in the bounds given by Lemma \ref{opcerjpellove} that satisfy equation (\ref{pelljdbaAB}). For $X_0=2$ we get only $Y_0=2$. But, then we would have $\frac{1}{2}(2b-2r)=b-r=b-2$, i.e.\@ $r=2$ which cannot be.
\smallskip \par
$2.)$ If $Y_0<0$, then $\frac{1}{2}(bX_0-r|Y_0|)=\frac{1}{2}(bX_0+rY_0)$ so we have $\frac{1}{2}(bX_0+rY_0)=2$.
 This implies $$bX_0-rY_0<2bX_0<2b\sqrt{\frac{a(b-a)}{r-2}}<2b\sqrt{b},$$ since $a<r-2$ (otherwise we would get $b\leq a+4$, which is in a contradiction with Lemma \ref{b_a_57sqrt}).
We also have that $4b(b-a)-4Y_0^2>4b^2-4ab-4b\sqrt{b}=4b(b-a-\sqrt{b}),$ since $Y_0^2<b^{3/2}$, so we can conclude
$$4=bX_0+rY_0>\frac{4b(b-a-\sqrt{b})}{bX_0-rY_0}>\frac{4b(b-a-\sqrt{b})}{2b\sqrt{b}}=\frac{2}{\sqrt{b}}(b-a-\sqrt{b}),$$
i.e.\@ $\sqrt{b}-1-\frac{a}{\sqrt{b}}<2$.\\
After squaring this expression and solving quadratic equation in $b$ we get $b<a+\frac{3}{2}(\sqrt{4a+9}+3)$. Again, by Lemma \ref{b_a_57sqrt} we also have $b\geq a+57\sqrt{a}$, and from these two inequalities we would get $a<1$, a contradiction.
\par
Hence, $h'$ must be even. From $Y_0 \equiv 2 (\bmod \ b)$ and $|Y_0|<b^{3/4}$ we conclude $Y_0=2$ and by direct computation from (\ref{pelljdbaAB}) we also get $X_0=2$.
\par
Now, we consider a system of equations (\ref{pelljdbaAC}) and (\ref{pelljdbaCD}). The proof is very similar to the previous system, so we omit details and only emphasize that here we use $c\geq a+b+2r$ to get a contradiction in the case that $j'$ is odd.
The same is used to prove that $k'$ is even when we consider the system of the equations (\ref{pelljdbaBC}) and (\ref{pelljdbaCD}).
\end{proof}

From the previous lemmas we see that equations (\ref{rjjdbaAB})-(\ref{rjjdbaCD}) actually have form:
\begin{align}
Y\sqrt{a}+X\sqrt{b}&=(2\sqrt{a}+2\sqrt{b})\left(\frac{r+\sqrt{ab}}{2}\right)^{2h}, \label{parnirjjdbaAB}\\
Z\sqrt{a}+X\sqrt{c}&=(2\sqrt{a}+2\sqrt{c})\left(\frac{s+\sqrt{ac}}{2}\right)^{2j},  \label{parnirjjdbaAC}\\
Z\sqrt{b}+Y\sqrt{c}&=(2\sqrt{b}+2\sqrt{c})\left(\frac{t+\sqrt{bc}}{2}\right)^{2k},  \label{parnirjjdbaBC}\\
W\sqrt{a}+X\sqrt{d}&=(2\varepsilon\sqrt{a}+2\sqrt{d})\left(\frac{x+\sqrt{ad}}{2}\right)^{2l},  \label{parnirjjdbaAD}\\
W\sqrt{b}+Y\sqrt{d}&=(2\varepsilon\sqrt{b}+2\sqrt{d})\left(\frac{y+\sqrt{bd}}{2}\right)^{2m}, \label{parnirjjdbaBD}\\
W\sqrt{c}+Z\sqrt{d}&=(2\varepsilon\sqrt{c}+2\sqrt{d})\left(\frac{z+\sqrt{cd}}{2}\right)^{2n}. \label{parnirjjdbaCD}
\end{align}

\section{Gap principle and classical congruences}
We have already observed, if there exist nonnegative integer $e$ such that the $D(4)$-quadruple $\{a,b,c,d\}$ can be extended to the quintuple $\{a,b,c,d,e\}$ then the equalities (\ref{parnirjjdbaAB})-(\ref{parnirjjdbaCD}) are satisfied for some nonnegative integers $(h,j,k,l,m,n)$.
We will now state and prove which relations hold between these indices, but first we will state without proof some known relations.
\begin{lemma}\cite[Lemma 5]{fil_xy4}\label{odnosi_indeksa_cetvorke}
If $Z=Z_{2j}^{(a,c)}=Z_{2k}^{(b,c)}$, then $k-1\leq j \leq 2k+1.$
\end{lemma}

\begin{lemma}\cite[Lemma 4]{konacno_petorki}\label{odnosim_n_l}
If $W=W_{2l}^{(a,d)}=W_{2m}^{(b,d)}=W_{2n}^{(c,d)}$,  then $8\leq n\leq m\leq l\leq 2n.$
\end{lemma}

As we can see, so far no one  has considered relations between $h$ and other indices. We will now prove which relation holds between $m$ and $h$ and improve the relation between $m$  and $l$.

\begin{lemma}\label{odnos_m_i_l}
We have $2l\leq 3m$ and $m<l$, for $m\geq 2$.
\end{lemma}
\begin{proof}
From (\ref{rjjdbaAD}) and (\ref{rjjdbaBD}) by expressing solutions explicitly we have
$$W_l^{(a,d)}=\frac{d+\varepsilon\sqrt{ad}}{\sqrt{ad}}\left(\frac{x+\sqrt{ad}}{2} \right)^l+ \frac{-d+\varepsilon\sqrt{ad}}{\sqrt{ad}}\left( \frac{x-\sqrt{ad}}{2} \right)^l,$$
and
$$W_m^{(b,d)}=\frac{d+\varepsilon\sqrt{bd}}{\sqrt{bd}}\left(\frac{y+\sqrt{bd}}{2} \right)^m+ \frac{-d+\varepsilon\sqrt{bd}}{\sqrt{bd}}\left( \frac{y-\sqrt{bd}}{2} \right)^m.$$
\par
Firstly, let us prove that $2l \leq 3m$ by observing that we must have $W_{2l}^{(a,d)}=W_{2m}^{(b,d)}$. \\
Notice that $x-\sqrt{ad}<1$, which implies that also $\left( \frac{x-\sqrt{ad}}{\sqrt{ad}} \right)^{2m}<1$ and since $-d+\varepsilon \sqrt{ad}<0$, we have
$$\frac{-d+\varepsilon \sqrt{ad}}{\sqrt{ad}}\left( \frac{x-\sqrt{ad}}{\sqrt{ad}} \right)^{2m}>\frac{-d+\varepsilon \sqrt{ad}}{\sqrt{ad}}\geq \frac{-d-\sqrt{ad}}{\sqrt{ad}}.$$
Moreover, notice that the second addend in the expressions for $W_l^{(a,d)}$ and $W_m^{(b,d)}$, respectively, is negative since  $d>b>a$, i.e.\@ $d>\sqrt{bd}>\sqrt{ad}$. Now, it is easy to see that
\begin{align*}
\frac{d+\varepsilon\sqrt{ad}}{\sqrt{ad}}\left(\frac{x+\sqrt{ad}}{2} \right)^{2l}- \frac{d+\sqrt{ad}}{\sqrt{ad}}&<W_{2l}^{(a,d)}=W_{2m}^{(b,d)}<\\
&<\frac{d+\varepsilon\sqrt{bd}}{\sqrt{bd}}\left(\frac{y+\sqrt{bd}}{2} \right)^{2m}.
\end{align*}
On the other hand,
$$\frac{d+\sqrt{ad}}{\sqrt{ad}}=1+\frac{d}{\sqrt{ad}}<\left(\frac{y+\sqrt{bd}}{2} \right)^2<\left(\frac{y+\sqrt{bd}}{2} \right)^{2m},$$
so we get the inequality
\begin{align*}
&\frac{d+\varepsilon\sqrt{ad}}{\sqrt{ad}}\left(\frac{x+\sqrt{ad}}{2} \right)^{2l}
<\left(\frac{d+\varepsilon\sqrt{bd}}{\sqrt{bd}}+1\right)\left(\frac{y+\sqrt{bd}}{2} \right)^{2m}\\
&\implies \left(\frac{x+\sqrt{ad}}{2} \right)^{2l}
<\sqrt{\frac{a}{b}} \cdot \frac{d+(\varepsilon+1)\sqrt{bd}}{d+\varepsilon\sqrt{ad}}\left(\frac{y+\sqrt{bd}}{2} \right)^{2m}.
\end{align*}
From $\sqrt{\frac{a}{b}}(d+(\varepsilon+1)\sqrt{bd})\leq \sqrt{\frac{a}{b}}(d+2\sqrt{bd})= \sqrt{\frac{a}{b}}d+2\sqrt{ad}<d+2\sqrt{ad}$ we have
$$\left(\frac{x+\sqrt{ad}}{2} \right)^{2l}
<\frac{d+2\sqrt{ad}}{d+\varepsilon\sqrt{ad}}\left(\frac{y+\sqrt{bd}}{2} \right)^{2m}.$$
Assume now the opposite, i.e. that $2l\geq 3m+1$. Then, we have
$$\left(\frac{x+\sqrt{ad}}{2} \right)^{3m+1}
<\frac{d+2\sqrt{ad}}{d+\varepsilon\sqrt{ad}}\left(\frac{y+\sqrt{bd}}{2} \right)^{2m}$$
and the inequality $\frac{x+\sqrt{ad}}{2}>\frac{d+2\sqrt{ad}}{d+\varepsilon\sqrt{ad}}$ implies
$$ \left(\frac{x+\sqrt{ad}}{2} \right)^{3}
<\left(\frac{y+\sqrt{bd}}{2} \right)^{2}.$$
Since $x+\sqrt{ad}>2\sqrt{ad}$ and $\sqrt{bd+4}<\sqrt{bd}+\frac{2}{\sqrt{bd}}$, we have $y+\sqrt{bd}<2\sqrt{bd}+\frac{2}{\sqrt{bd}}<2\sqrt{bd}\left( 1+\frac{1}{bd} \right)<2\sqrt{bd}\left( 1+\frac{1}{B_0^3} \right)$  where $B_0<b$. Now we get to observe the inequality
$$(\sqrt{ad})^3<(\sqrt{bd})^2\left( 1+\frac{1}{B_0^3} \right)^2$$
and after squaring, inserting $abc<d$ and canceling we get
$$a^4c<b\left( 1+\frac{1}{B_0^3} \right)^4.$$
For $B_0=10^5$, we see that the inequality cannot be true for $a>1$ or for $c>4b$. It remains to observe the case where $a=1$ and $c=a+b+2r$. In this case we have
$$1+b+2\sqrt{b+4}<b\left( 1+\frac{1}{B_0^3}\right)^4$$
and for $B_0=10^5$ we get $b>2.5\cdot 10^{29}$ and that value can be used as a new value for $B_0$. After inserting this value we get an inequality which doesn't have solutions $b$ in positive integers. So each case leads to a contradiction, which implies that our assumption was wrong, i.e.\@ we have $2l \leq 3m$.
\par
Now we assume $m=l$. Similarly as before, we observe that we have
\begin{align*}
\frac{d+\varepsilon\sqrt{bd}}{\sqrt{bd}}\left( \frac{y+\sqrt{bd}}{2}\right)^{2m}-\frac{d+\sqrt{bd}}{\sqrt{bd}}&<W_{2m}^{(b,d)}=W_{2m}^{(a,d)}\\
&<\frac{d+\varepsilon\sqrt{ad}}{\sqrt{ad}}\left( \frac{x+\sqrt{ad}}{2}\right)^{2m}
\end{align*}
and since $d+\sqrt{bd}<\left( \frac{y+\sqrt{bd}}{2}\right)^2$, we get
$$\frac{d+\varepsilon\sqrt{bd}-1}{\sqrt{bd}}\left( \frac{y+\sqrt{bd}}{2}\right)^{2m}<\frac{d+\varepsilon\sqrt{ad}}{\sqrt{ad}}\left( \frac{x+\sqrt{ad}}{2}\right)^{2m}$$
and after multiplying and rearranging we get
$$\left( \frac{y+\sqrt{bd}}{x+\sqrt{ad}}\right)^{2m}<\sqrt{\frac{b}{a}}\cdot \frac{d+\varepsilon\sqrt{ad}}{d+\varepsilon\sqrt{bd}-1}.$$
But, we have
\begin{align*}
 \frac{d+\varepsilon\sqrt{ad}}{d+\varepsilon\sqrt{bd}-1}&<\frac{d+\sqrt{bd}}{d-\sqrt{bd}}=1+\frac{2\sqrt{bd}}{d-\sqrt{bd}}=1+\frac{2}{\frac{d}{\sqrt{bd}}-1}\\
 &<1+\frac{2}{\sqrt{B_0}-1}=\frac{\sqrt{B_0}+1}{\sqrt{B_0}-1},
\end{align*}
where the last inequality is true since $\frac{d}{\sqrt{bd}}=\sqrt{\frac{d}{b}}>\sqrt{\frac{abc}{b}}=\sqrt{ac}>\sqrt{b}>\sqrt{B_0}$.
So, it must hold
$$\left( \frac{y+\sqrt{bd}}{x+\sqrt{ad}}\right)^{2m}<\sqrt{\frac{b}{a}}\cdot \frac{\sqrt{B_0}+1}{\sqrt{B_0}-1}.$$
On the other hand, it is easy to see that
$$\left( \frac{y+\sqrt{bd}}{x+\sqrt{ad}}\right)^{2}>\sqrt{\frac{b}{a}},$$
which, with the previous inequality, leads to the conclusion that
$$\sqrt{\frac{b}{a}}^{m-1}< \frac{\sqrt{B_0}+1}{\sqrt{B_0}-1}.$$
From \cite[Lemma 3.2]{nas} we can conclude that $l'=2l>0.61803d^{1/4}>0.61803\cdot 10^{10/4}>195$. Also, from Lemma \ref{odnosi_indeksa_cetvorke} we have $2l\leq 4m+1$ so $m>48$. Now we observe 
$$\sqrt{\frac{b}{a}}^{47}< \frac{\sqrt{B_0}+1}{\sqrt{B_0}-1}$$
i.e.\@ $$(a+57\sqrt{a})^{47/2}<b^{47/2}<\frac{\sqrt{B_0}+1}{\sqrt{B_0}-1}a^{47/2}$$
and by solving this inequality in $a$ for $B_0=10^5$ we obtain $a>4.484\cdot 10^{10}$ which can be used as a new value for $B_0$, since $b>a$. By iterating this process we get a contradiction, this time a contradiction is with the upper bound $b<10^{36}$ from \cite{nas}. We can now conclude $m\neq l$.
\end{proof}

\begin{lemma}\label{odnos_h_im}
We have $h\geq 2m$.
\end{lemma}
\begin{proof}
Similarly as in the previous Lemma, for sequences $Y_{2h}^{(a,b)}$ and $Y_{2m}^{(b,d)}$ we have
\begin{align*}
&Y_{2h}^{(a,b)}=\frac{b+\sqrt{ab}}{\sqrt{ab}}\left( \frac{r+\sqrt{ab}}{2}\right)^{2h}+\frac{-b+\sqrt{ab}}{\sqrt{ab}}\left( \frac{r-\sqrt{ab}}{2}\right)^{2h},\\
&Y_{2m}^{(b,d)}=\frac{\varepsilon b+\sqrt{bd}}{\sqrt{bd}}\left( \frac{y+\sqrt{bd}}{2}\right)^{2m}+\frac{-\varepsilon b+\sqrt{bd}}{\sqrt{bd}}\left( \frac{y-\sqrt{bd}}{2}\right)^{2m}.\\
\end{align*}
If $Y=Y_{2h}^{(a,b)}=Y_{2m}^{(b,d)}$, we have
\begin{align*}
&\left(1-\sqrt{b/d}\right)\left( \frac{y+\sqrt{bd}}{2}\right)^{2m}<\frac{-b+\sqrt{bd}}{\sqrt{bd}}\left( \frac{y+\sqrt{bd}}{2}\right)^{2m}<Y_{2m}^{(b,d)}=Y_{2h}^{(a,b)}<\\&<\frac{b+\sqrt{ab}}{\sqrt{ab}}\left(\frac{r+\sqrt{ab}}{2}\right)^{2h}\leq \left(1+\sqrt{b/a}\right)\left(\frac{r+\sqrt{ab}}{2}\right)^{2h}.
\end{align*}
It is easy to see that $\frac{\sqrt{d}(\sqrt{a}+\sqrt{b})}{\sqrt{a}(\sqrt{d}-\sqrt{b})}<\frac{r+\sqrt{ab}}{2}$,
so we have
$$\left( \frac{y+\sqrt{bd}}{2}\right)^{2m}<\frac{1+\sqrt{b/a}}{1-\sqrt{b/d}}\left(\frac{r+\sqrt{ab}}{2}\right)^{2h}<\left(\frac{r+\sqrt{ab}}{2}\right)^{2h+1}.$$
Since
$$\frac{y+\sqrt{bd}}{2}>\sqrt{bd}>\sqrt{ab^2c}\geq \sqrt{ab^2(a+b+2r)}>r^2>\left(\frac{r+\sqrt{ab}}{2}\right)^2,$$
we get
$$\left(\frac{r+\sqrt{ab}}{2}\right)^{4m}<\left( \frac{y+\sqrt{bd}}{2}\right)^{2m}<\left(\frac{r+\sqrt{ab}}{2}\right)^{2h+1},$$
where it is easy to conclude $4m<2h+1$, i.e.\@ $2m\leq h$.
\end{proof}

For the completeness we will state classic congruences that hold for $D(4)$-quintuple.
\begin{lemma}\label{kongruencije}
Let $\{a,b,c,d,e\}$ be a $D(4)$-quintuple. Then
$$a\varepsilon l^2+xl\equiv b\varepsilon m^2+ym \equiv c\varepsilon n^2+zn (\bmod \ d).$$
\end{lemma}
\begin{proof}
If we observe the sequence $W_{2l}^{(a,d)}$ we see that
\begin{align*}
W_{2l+2}&=xW_{2l+1}-W_{2l}=x^2W_{2l}-(xW_{2l-1}-W_{2l-2})-W_{2l-2}-W_{2l}=\\
&=x^2W_{2l}-W_{2l}-W_{2l-2}-W_{2l}=(x^2-2)W_{2l}-W_{2l-2}.
\end{align*}
As in \cite[Lemma 3]{dujram} it is easy to prove
$$W_{2l}^{(a,d)}\equiv 2\varepsilon +d(a\varepsilon l^2+xl) (\bmod \ d^2),$$
and since $W=W_{2l}^{(a,d)}=W_{2m}^{(b,d)}=W_{2n}^{(c,d)}$, and analogous results hold for all sequences, for a $D(4)$-quintuple, we get
$$a\varepsilon l^2+xl\equiv b\varepsilon m^2+ym \equiv c\varepsilon n^2+zn (\bmod \ d).$$
\end{proof}

Unfortunately, using these congruences and methods from \cite{petorke} we could not get $m>\alpha \sqrt{d/b}$ for some coefficient $\alpha$ as "large" as the one proved for $D(1)$-quintuples in \cite{petorke}. Our largest possible $\alpha$ was obtained after adjusting the method from \cite[Proposition 3.1]{ct}, which we have also used in \cite{nas} to get a similar coefficient for $D(4)$-quadruples. We omit the proof since it is similar to the one given in detail in \cite{ct}.

\begin{lemma}\label{lbid}
Let $\{a,b,c,d,e\}$ be a $D(4)$-quintuple such that $a<b<c<d<e$, $W_{2l}^{(a,b)}=W_{2m}^{(b,d)}$ and $\frac{3}{2}m\geq l>m\geq 2$. Assume that $a\geq A_0$, $b\geq B_0$ and $d\geq D_0$, $b>\rho a$, $\rho \geq 1$. Then
$$l>\alpha b^{-1/2}d^{1/2}$$
for every real number $\alpha$ that satisfy both inequalities
\begin{align}
\alpha^2 +\alpha(1+2B_0^{-1}D_0^{-1})&\leq 1, \label{prvanej}\\
\frac{20}{9}\alpha^2+\alpha(B_0(\lambda+\rho^{-1/2})+2D_0^{-1}(\lambda+\rho^{1/2}))&\leq B_0, \label{druganej}
\end{align}
where $\lambda=\sqrt{\frac{A_0+4}{\rho A_0+4}}.$
\end{lemma}

Now we use this result to get lower bounds on indices in the terms of $ac$.

\begin{lemma}\label{indeksi_ac}
Let $\{a,b,c,d,e\}$ be a $D(4)$-quintuple. Then $l>0.499997\sqrt{ac}$, $j>m>0.333331\sqrt{ac}$ and $h>0.666662\sqrt{ac}$.
\end{lemma}
\begin{proof}
By inserting $\rho=1$, $A_0=1$, $B_0=10^5$ and $D_0=10^{10}$ in the inequalities from Lemma \ref{lbid} we compute that $\alpha=0.499997$. The statement now follows from Lemmas \ref{odnosim_n_l},  \ref{odnos_m_i_l} and \ref{odnos_h_im} and the fact that $d>abc$.
\end{proof}

\section{Linear forms in logarithms}
In this section we use different methods to find a good upper bound on the index $h$ and a product $ac$ in a $D(4)$-quintuple. Even though many authors usually apply Matveev's theorem on a linear form in logarithms, we will use Aleksentsev's version of the theorem from \cite{aleks} as authors in \cite{ct} did and which we also applied in \cite{nas} because it will give us slightly better bounds. \par
For any non-zero algebraic number $\gamma$ of degree $D$ over $\mathbb{Q}$, with minimal polynomial $A\prod_{j=1}^{D}\left( X-\gamma^{(j)} \right)$ over $\mathbb{Z}$, we define its absolute logarithmic height as
$$h(\gamma)=\frac{1}{D}\left( \log A+\sum_{j=1}^D \log^+ \left|(\gamma^{(j)})\right| \right), $$
where $\log^+\alpha=\log \max \left\{ 1, \alpha \right\}.$

\begin{theorem}[Aleksentsev]\label{aleks}
Let $\Lambda$ be a linear form in logarithms of $n$ multiplicatively independent totally real algebraic numbers $\alpha_1,\dots, \alpha_n$, with rational coefficients $b_1,\dots,b_n$. Let $h(\alpha_j)$ denote the absolute logarithmic height of $\alpha_j$ for $1\leq j\leq n$. Let $d$ be the degree of the number field $\mathcal{K}=\mathbb{Q}(\alpha_1,\dots,\alpha_n)$, and let $A_j=\max \left(dh(\alpha_j),|\log \alpha_j|,1 \right)$. Finally, let
\begin{equation}
E=\max \left( \max_{1\leq i,j \leq n} \left\{ \frac{|b_i|}{A_j}+\frac{|b_j|}{A_i} \right\},3 \right).
\end{equation}
Then
$$\log|\Lambda|\geq -5.3n^{\frac{1-2n}{2}}(n+1)^{n+1}(n+8)^2(n+5)31.44^n d^2(\log E)A_1\cdots A_n\log(3nd).$$
\end{theorem}

\par
Let us define a linear form in logarithms
$$\Lambda_1:=2h\log\frac{r+\sqrt{ab}}{2}-2j\log\frac{s+\sqrt{ac}}{2}+\log{\frac{\sqrt{c}(\sqrt{a}+\sqrt{b})}{\sqrt{b}(\sqrt{a}+\sqrt{c})}}.$$
Analogously as in \cite[Lemma 17]{petorke} we can find the bounds for $\Lambda_1$.
\begin{lemma}\label{lin_forma_granice}
We have $0<\Lambda_1<\left(\frac{s+\sqrt{ac}}{2}\right)^{-4j}.$
\end{lemma}
To apply Theorem \ref{aleks}\@ first we must find values of the parameters, and we can easily see that
\begin{align*}
&n=3,\quad d=4, \quad b_1=2h,\quad b_2=-2j,\quad b_3=1;\\
&\alpha_1=\frac{r+\sqrt{ab}}{2},\quad \alpha_2=\frac{s+\sqrt{ac}}{2},\quad \alpha_3=\frac{\sqrt{c}(\sqrt{a}+\sqrt{b})}{\sqrt{b}(\sqrt{a}+\sqrt{c})}.
\end{align*}
It is not difficult to see that $h(\alpha_1)=\frac{1}{2}\log \alpha_1$ and $h(\alpha_2)=\frac{1}{2}\log \alpha_2$.\\
Minimal polynomial of $\alpha_3$ is equal to a polynomial
\begin{align*}
p_3(X)&=b^2(c-a)^2X^4-4b^2c(c-a)X^3+\\
&2bc(3bc-a^2-ac-ab)X^2-4bc^2(b-a)X+c^2(b-a)^2
\end{align*}
divided by the greatest common divisor of its coefficients, which we will denote with $g$. Zeros of the polynomial $p_3(X)$ are $\beta_1=\frac{\sqrt{c}(-\sqrt{a}+\sqrt{b})}{\sqrt{b}(\sqrt{a}+\sqrt{c})}$, $\beta_2=\frac{\sqrt{c}(\sqrt{a}+\sqrt{b})}{\sqrt{b}(-\sqrt{a}+\sqrt{c})}$, $\beta_3=\frac{\sqrt{c}(-\sqrt{a}+\sqrt{b})}{\sqrt{b}(-\sqrt{a}+\sqrt{c})}$ and $\alpha_3$. It holds
$$\beta_1<\beta_3<1$$
and
$$1<\alpha_3<\beta_2,$$
which implies
$$h(\alpha_3)=\frac{1}{4}\left(\log \frac{b^2(c-a)^2}{g}+\log \alpha_3 +\log \beta_2 \right)\leq\frac{1}{4}\left(\log (b^2(c-a)^2)+\log \alpha_3 +\log \beta_2 \right).$$
We can observe that
\begin{align*}
h(\alpha_3)&\leq \frac{1}{4}\left(\log (b^2(c-a)^2)+\log \frac{c(\sqrt{a}+\sqrt{b})^2}{b(c-a)}\right)\\
&=\frac{1}{4}\log(cb(c-a)(\sqrt{a}+\sqrt{b})^2)\\
&<\frac{1}{4}\log c^4= \log c.
\end{align*}
Since the function on the right hand side of the inequality in Theorem \ref{aleks}\@ is decreasing in $A_3$ we can take
$$A_1=4\frac{1}{2}\log \alpha_1=2\log \alpha_1,\quad A_2=2\log \alpha_2, \quad A_3=4\log c=\log c^4 .$$
\par
Observe that $A_1<A_2<A_3$ and $j<h$, so we  have
$E=\max\left\{\frac{2h}{\log \alpha_1},3 \right\}.$
Since $0.66\sqrt{ac}>0.66 r>\log r^3 $, which is true for every $r>10$, we have $h>0.66\sqrt{ac}>3\log r> 3\log \alpha_1$ which implies $\frac{2h}{\log \alpha_1}>3$, i.e.\@ we can take $E=\frac{2h}{\log \alpha_1}$ and apply Theorem \ref{aleks}\@ to get,
\begin{align*}
\log|\Lambda_1|>&-5.3n^{0.5-n}(n+1)^{n+1}(n+8)^2(n+5)31.44^n d^2\log \frac{2h}{\log \alpha_1}\\
& \cdot 2\log\alpha_1 \cdot 2 \log \alpha_2 \cdot 4 \log c\cdot \log(3nd).
\end{align*}
On the other hand, from Lemma \ref{lin_forma_granice}\@ and the fact that $|b_1|A_1<|b_2|A_2$ we have
$$\log|\Lambda_1|<-4j\log \alpha_2<-4h \log \alpha_1,$$
which now implies
\begin{align*}
4h \log \alpha_1<&5.3n^{0.5-n}(n+1)^{n+1}(n+8)^2(n+5)31.44^n d^2\log \frac{2h}{\log \alpha_1}\\
& \cdot 2\log\alpha_1 \cdot 2 \log \alpha_2 \cdot 4 \log c\cdot \log(3nd).
\end{align*}
We put $n=3$,  $d=4$ and get
$$\frac{h}{\log 2h-\log\log\sqrt{10^5}}<6.005175\cdot 10^{11} \log \alpha_2 \log c,$$
where we have used $\alpha_1>\sqrt{ab}>\sqrt{10^5}$.
Now we use that $\alpha_2<\sqrt{ac+4}$, $c\leq ac$ and since the left hand side of the inequality is increasing in $h$ we can use $h>0.666662\sqrt{ac}$ to get
\begin{equation}\label{ac_aleks}
ac<1.08915\cdot 10^{34}
\end{equation}
and
\begin{equation}\label{h_aleks}
h<6.95745\cdot 10^{16}.
\end{equation}
\par
We collect these observations in the next Proposition.
\begin{proposition}\label{matveev_prop}
Let  $\{a,b,c,d,e\}$ be a $D(4)$-quintuple such that $a<b<c<d<e$, then $ac<1.08915\cdot 10^{34}$ and $h<6.95745\cdot 10^{16}$.
Moreover, $$\frac{h}{\log 2h-\log\log\sqrt{10^5}}<6.005175\cdot 10^{11} \log \alpha_2 \log c.$$
\end{proposition}

To get a sharper bound on $ac$ and $h$, which we need later, we will use the Proposition \ref{matveev_prop}\@ together with a tool due to Mignotte \cite{mignote} and then on some of the cases, we will use Laurent's theorem. First, we will state Mignotte's theorem and show how can it be applied to $D(4)$-quintuples. We aim to give the most general algorithm to find appropriate parameters, so it can be clear how we can easily repeat the procedure multiple times to get better results.

\begin{theorem}[Mignotte]\label{mignote}
We observe three non-zero algebraic numbers $\alpha_1$, $\alpha_2$ and $\alpha_3$, which are either all real and greater than $1$ or all complex of modulus one and all different from $1$. Moreover, we assume that either the three numbers $\alpha_1$, $\alpha_2$ and $\alpha_3$ are multiplicatively independent, or two of these numbers are multiplicatively independent and the third one is a root of unity. Put
$$\mathcal{D}=[\mathbb{Q}(\alpha_1,\alpha_2,\alpha_3):\mathbb{Q}]/[\mathbb{R}(\alpha_1,\alpha_2,\alpha_3):\mathbb{R}].$$
We also consider three positive coprime rational integers $b_1,b_2,b_3$, and the linear form
$$\Lambda=b_2\log \alpha_2-b_1\log \alpha_1-b_3\log\alpha_3,$$
where the logarithms of $a_i$ are arbitrary determinations of the logarithm, but which are all real or all purely imaginary.
And we assume also that
$$b_2|\log \alpha_2|=b_1|\log \alpha_1|+b_3|\log\alpha_3| \pm |\Lambda|.$$
We put
$$d_1=\gcd(b_1,b_2),\: d_3=\gcd(b_3,b_2),\: b_1=d_1b_1',\: b_2=d_1b_2'=d_3b_2'',\: b_3=d_3b_3''. $$
Let $\rho>e$ be a real number and put $\lambda=\log \rho$. Let $a_1$, $a_2$ and $a_3$ be real numbers such that
$$a_i\geq \rho |\log \alpha_i|-\log |\alpha_i|+2\mathcal{D} h(\alpha_i),\quad i=1,2,3,$$
and assume further that
$$\Omega:=a_1a_2a_3\geq 2.5,\quad \textit{and}\quad A:=\min\{a_1,a_2,a_3\}\geq 0.62.$$
Let $K$, $L$ and $M$ be positive integers with
$$L\geq 4+\mathcal{D},\quad K=\lfloor M\Omega L \rfloor,\quad \textit{where } \,M\geq 3.$$
Let $0<\chi\leq 2$ be fixed.  Define
\begin{align*}
c_1&=\max\left\{ (\chi ML)^{2/3},\sqrt{2ML/A}\right\},\\
c_2&=\max\left\{2^{1/3}(ML)^{2/3}, \sqrt{M/A}L \right\},\\
c_3&=(6M^2)^{1/3}L,
\end{align*}
and then put
\begin{align*}
R_1&=\lfloor c_1a_2a_3 \rfloor,\quad S_1=\lfloor c_1a_1a_3 \rfloor, \quad T_1=\lfloor c_1a_1a_2 \rfloor ,\\
R_2&=\lfloor c_2a_2a_3 \rfloor,\quad S_2=\lfloor c_2a_1a_3 \rfloor, \quad T_2=\lfloor c_2a_1a_2 \rfloor ,\\
R_3&=\lfloor c_3a_2a_3 \rfloor,\quad S_3=\lfloor c_3a_1a_3 \rfloor, \quad T_3=\lfloor c_3a_1a_2 \rfloor .
\end{align*}
Let also
$$R=R_1+R_2+R_3+1,\quad S=S_1+S_2+S_3+1,\quad T=T_1+T_2+T_3+1.$$
Define
$$c_0=\max\left\{ \frac{R}{La_2a_3},\frac{S}{La_1a_3},\frac{T}{La_1a_2}\right\}.$$
Finally, assume that
\begin{multline}\label{mignote_nejednakost}
\left(\frac{KL}{2}+\frac{L}{4}-1-\frac{2K}{3L} \right)\lambda+2\mathcal{D}\log 1.36\\
\geq (\mathcal{D}+1)\log L+3gL^2c_0\Omega+\mathcal{D}(K-1)\log \widetilde{b} +2\log K,
\end{multline}
where
$$g=\frac{1}{4}-\frac{K^2L}{12RST},\quad b'=\left(\frac{b_1'}{a_2}+\frac{b_2'}{a_1} \right)\left( \frac{b_3''}{a_2}+\frac{b_2''}{a_3}\right),\quad \widetilde{b}=\frac{e^3c_0^2\Omega^2L^2}{4K^2}\times b'.$$
Then either
\begin{equation}\label{mignote_lambda_ocjena}
\log|\Lambda|>-(KL+\log(3KL))\lambda,
\end{equation}
\begin{description}
\item [or (A1)] there exist two non-zero rational integers $r_0$ and $s_0$ such that
$$r_0b_2=s_0b_1$$
with
$$|r_0|\leq \frac{(R_1+1)(T_1+1)}{\mathcal{M}-T_1}\quad \textit{and}\quad|s_0|\leq  \frac{(S_1+1)(T_1+1)}{\mathcal{M}-T_1} $$
where
\begin{align*}
\mathcal{M}&=\max\{R_1+S_1+1,S_1+T_1+1,R_1+T_1+1,\chi \mathcal{V} \},\\
\mathcal{V} &=\sqrt{(R_1+1)(S_1+1)(T_1+1)},
\end{align*}

\item [or(A2)] there exist rational integers $r_1$, $s_1$, $t_1$ and $t_2$, with $r_1s_1\neq 0$ such that
$$(t_1b_1+r_1b_3)s_1=r_1b_2t_2,\quad \gcd(r_1,t_1)=\gcd(s_1,t_2)=1,$$
which also satisfy
\begin{align*}
|r_1s_1|&\leq \delta \cdot \frac{(R_1+1)(S_1+1)}{\mathcal{M}-\max\{R_1,S_1\}},\\
|s_1t_1|&\leq \delta \cdot \frac{(S_1+1)(T_1+1)}{\mathcal{M}-\max\{S_1,T_1\}},\\
|r_1t_2|&\leq \delta \cdot \frac{(R_1+1)(T_1+1)}{\mathcal{M}-\max\{R_1,T_1\}},
\end{align*}
where $\delta=\gcd(r_1,s_1)$. Moreover, when $t_1=0$ we can take $r_1=1$, and  when $t_2=0$ we can take $s_1=1$.
\end{description}
\end{theorem}

We consider the linear form
$$\Lambda=-\Lambda_1=2j\log \alpha_2-2h \log \alpha_1 -\log \alpha_3.$$
It is important to notice that we have $c>b>10^5$.\\
As before we have
$$\mathcal{D}=4,\quad b_1=2h,\quad b_2=2j,\quad b_3=1,$$
and we can again take
$$h(\alpha_1)=\frac{1}{2}\log \alpha_1,\quad h(\alpha_2)=\frac{1}{2}\log \alpha_2, \quad h(\alpha_3)<\log c.$$
\\
Observe that
$$\log \alpha_3<\log\left(1+\sqrt{\frac{a}{b}} \right)<\log 2<0.694.$$
Now we have to choose $a_i\geq \rho |\log \alpha_i|-\log |\alpha_i|+2\mathcal{D} h(\alpha_i)$ for $i\in\{1,2,3\}$. In each case we have $|\log \alpha_i|=\log |\alpha_i|=\log \alpha_i$. Let $i=1$, then
$$
a_1\geq \rho \log \alpha_1-\log \alpha_1+ 4\cdot \log \alpha_1=(\rho+3)\log \alpha_1
$$
and similar observation is true for $i=2$. For $i=3$ we have
$$
a_3\geq \rho \log \alpha_3-\log \alpha_3+2\cdot 4\cdot \log c,
$$
so we see that we can take
\begin{align*}
a_1&=(\rho+3)\log \alpha_1\\
a_2&=(\rho+3)\log \alpha_2\\
a_3&=8(\log c+0.08675(\rho-1)).
\end{align*}
For the simplicity of the proof we will give intervals for parameters $M$, $L$ and $\rho$, but we will not give their explicit values, because we will search within these intervals to find the values which give us the best possible bound on index $h$. From now on, when ever is needed, we assume that $\chi=2$, $\rho \in [5.5,14]$, $ L \in [700,1500]$, and $M\in [3,10]$. These intervals were chosen since they seemed sufficient, after observing some random values, for finding an optimal value for upper bound on $h$ and also because they satisfy all conditions needed, as we will prove.
\\ Now, let us observe which conditions these parameters must satisfy so we can use Theorem \ref{mignote}.
\par
It is easy to see that we always have $a_1<a_2$, so $A=\min\{a_1,a_2,a_3\}=\min\{a_1,a_3\}$. If $A=a_1$ we have $A=(\rho+3)\log \alpha_1>5\log \sqrt{ab}$, and if $A=a_3$ then $A>8\log c$, so in either case it is $A\geq 0.62$. \\
Moreover, it is also easy to see that we always have $\Omega=a_1\cdot a_2 \cdot a_3 >2.5$. \\
Values $c_1$, $c_2$ and $c_3$ can easily be calculated for specific values of the parameters.

We get an upper bound for $c_0$ after observing that
$$
\frac{R}{La_2a_3}=\frac{R_1+R_2+R_3+1}{La_2a_3}<\frac{c_1+c_2+c_3+1}{L}
$$
and since the same is true for $S$ and $T$, we have $c_0<\frac{c_1+c_2+c_3+1}{L}.$\\
Also  $$\Omega=a_1a_2a_3=8(\rho+3)^2\log \alpha_1 \log \alpha_2 (\log c+0.08675(\rho-1))$$ and
$$K=\lfloor M\Omega L \rfloor=\lfloor 8ML(\rho+3)^2\log \alpha_1 \log \alpha_2 (\log c+0.08675(\rho-1))\rfloor. $$
To see when inequality (\ref{mignote_nejednakost}) holds, let us observe it by parts:
\\
We have $M \Omega L-1< K\leq M\Omega L$ so
\begin{align*}
&\left(\frac{KL}{2}+\frac{L}{4}-1-\frac{2K}{3L} \right)\lambda+2\mathcal{D}\log 1.36\\
&> M\Omega L\left(\frac{L}{2}-\frac{2}{3L} \right)\lambda-\left(\frac{L}{2}-\frac{2}{3L} \right)\lambda+\left( \frac{L}{4}-1\right)\lambda+2\mathcal{D}\log 1.36\\
&=8ML(\rho+3)^2\left(\frac{L}{2}-\frac{2}{3L} \right)\lambda \log \alpha_1 \log \alpha_2 \log c\\
&+8ML(\rho+3)^2\left(\frac{L}{2}-\frac{2}{3L} \right)\lambda\cdot  0.08675(\rho-1) \log \alpha_1 \log \alpha_2 \\
&+\left( \frac{L}{4}-1\right)\lambda+2\mathcal{D}\log 1.36-\left(\frac{L}{2}-\frac{2}{3L} \right)\lambda.
\end{align*}
On the other hand, for the expressions on the right hand side of the inequality (\ref{mignote_nejednakost}) it holds:
\begin{enumerate}
\item Since we can use $ac<1.08915\cdot10^{34}$ we get a numerical value
$$(\mathcal{D}+1)\log L+2\log K\leq 5\log L+2\log (8ML(\rho+3)^2\log ^2 \sqrt{ac+4}\log ac).$$
\item Also, from $g=\frac{1}{4}-\frac{K^2L}{12RST}<\frac{1}{4}$ we get
\begin{align*}
3gL^2c_0\Omega<\frac{3}{4}L^2c_0\Omega&\leq  \frac{3}{4}L^2c_0\cdot8(\rho+3)^2 \log \alpha_1 \log \alpha_2 \log c\\
&+ \frac{3}{4}L^2c_0\cdot8\cdot 0.08675(\rho-1)(\rho+3)^2 \log \alpha_1 \log \alpha_2 .
\end{align*}
\item To approximate the last part of the right hand side of the inequality, observe that from $\log \alpha_3<2\log \alpha_1$, since
$\Lambda_1>0$, we have $2(h+1)\log \alpha_1-2j\log \alpha_2>0$, i.e.
$$\frac{b_2}{a_1}<\frac{b_1+2}{a_2}.$$
Also, since $2\log \alpha_2>\log c$ and $\rho\geq 5.5$, we have $\frac{b_3}{a_2}<\frac{2}{a_3}$ and since $j<h$ we get
$$
b'<\frac{(4h+2)(2h+2)}{8(\rho+3)\log \alpha_2 \log c}.
$$
Using $c>10^5$, $h<6.95745\cdot 10^{16}$ and values of the parameters, we can calculate an upper bound for $b'$.\\
Then we have
$$\frac{K}{\Omega}>\frac{M\Omega L-1}{\Omega}>ML-1$$
and
$$\log \widetilde{b}<\log \left(\frac{c_0^2}{4}e^3\frac{1}{(ML-1)^2}L^2b' \right).$$
Finally,
\begin{align*}
\mathcal{D}(K-1)\log \widetilde{b}&<4M\Omega L \log \widetilde{b}\\
&=32ML(\rho+3)^2\log \widetilde{b} \log \alpha_1 \log \alpha_2\log c\\
&+32ML(\rho+3)^2\log \widetilde{b}\cdot 0.08675(\rho-1) \log \alpha_1 \log \alpha_2.
\end{align*}
\end{enumerate}
\par As we can see from above, we have expressions of the form  $\log \alpha_1 \log \alpha_2\log c$, $\log \alpha_1 \log \alpha_2$ and numerical values, and to see if some selected values of the parameters $M$, $L$ and $\rho$ satisfy inequality (\ref{mignote_nejednakost}) it is enough to compare coefficients of these expressions.
For each selection of values for the parameters $M$, $L$ and $\rho$ which satisfy these condition, we can apply Theorem \ref{mignote}\@ and have that either cases  \textit{(A1)} or \textit{(A2)} hold or inequality (\ref{mignote_lambda_ocjena}) holds. Let us first observe this inequality.
We then have
\begin{align*}
\log|-\Lambda_1|&>-(KL+\log(3KL))\lambda\\
&\geq -(ML^2\Omega+\log(3ML^2\Omega))\log \rho,
\end{align*}
and on the other hand,
$$\log|-\Lambda_1|<-4j\log \alpha_2<-4h\log \alpha_1$$
which can be proven by using Lemma \ref{lin_forma_granice}\@, so
$$4h\log \alpha_1<(ML^2\Omega+\log(3ML^2\Omega))\log \rho.$$
Notice that $ML^2\Omega>8ML^2(\rho+3)^2\log \sqrt{ab}\log \sqrt{ac}\log c>3.81\cdot 10^{10}$, and for  $x>3.81\cdot 10^{10}$ we have $\log 3x<6.7\cdot 10^{-10}x$, so we can observe
$$4h\log \alpha_1<ML^2\Omega(1+6.7\cdot 10^{-10})\log \rho,$$ i.e.\@
$$h<2ML^2(\rho+3)^2\log \rho (1+6.7\cdot 10^{-10})\left(1+\frac{0.08675}{\log 10^5}(\rho-1) \right)\log \alpha_2 \log c.$$
From now on, to shorten an expression $x$, with $G(x)$ we will denote upper bound for the numerical value we get by inserting all parameters in the expression except those which contain values of a triple $\{a,b,c\}$. In this expression with $G(h)$ we denote
$$G(h):=2ML^2(\rho+3)^2\log \rho (1+6.7\cdot 10^{-10})\left(1+\frac{0.08675}{\log 10^5}(\rho-1) \right),$$
so we have $h<G(h)\cdot \log \alpha_2 \log c.$
\par If the inequality (\ref{mignote_lambda_ocjena}) does not hold, then one of the cases \textit{(A1)} or \textit{(A2)} holds. \\
Notice that  $\mathcal{M}>\chi \mathcal{V}>\chi c_1^{3/2}a_1a_2a_3.$ For each $a_i$ we calculate the lower bounds
\begin{align*}
&a_2>a_1>(\rho+3)\log 10^{5/2}:=A_{1,2},
&a_3>8(\log 10^5+0.08675(\rho-1)):=A_3.
\end{align*}
Observe that since $a_2>a_1$ then $\max\{R_1,S_1\}=R_1$, but values of $\max\{S_1,T_1\}$ and $\max\{R_1,T_1\}$ depend on the values of a triple $\{a,b,c\}$, so we must address these cases separately.\\
Let us denote and observe
\begin{align*}
B_1&:=\frac{(R_1+1)(S_1+1)}{\mathcal{M}-\max\{R_1,S_1\}}<\frac{(c_1a_2a_3+1)(c_1a_1a_3+1)}{\chi c_1^{3/2}a_1a_2a_3-c_1a_2a_3}\\
&=\frac{1+\frac{1}{c_1a_2a_3}}{\frac{\chi}{2}-\frac{1}{2c_1^{1/2}a_1}}\left(0.5c_1^{1/2}+\frac{1}{2c_1^{1/2}a_1a_3} \right)a_3\\
&<\frac{1+\frac{1}{c_1A_{1,2}A_3}}{\frac{\chi}{2}-\frac{1}{2c_1^{1/2}A_{1,2}}}\left(0.5c_1^{1/2}+\frac{1}{2c_1^{1/2}A_{1,2}A_3} \right)8\left( 1+ \frac{0.08675}{\log 10^5}(\rho-1) \right)\log c\\
&=:G(B_1)\cdot \log c.
\end{align*}
Let us assume that $\max\{S_1,T_1\}=S_1$. Then
\begin{align*}
B_2&:=\frac{(S_1+1)(T_1+1)}{\mathcal{M}-\max\{S_1,T_1\}}<\frac{(c_1a_1a_3+1)(c_1a_1a_2+1)}{\chi c_1^{3/2}a_1a_2a_3-c_1a_1a_3}\\
&=\frac{1+\frac{1}{c_1a_1a_3}}{\frac{\chi}{2}-\frac{1}{2c_1^{1/2}a_2}}\left(0.5c_1^{1/2}+\frac{1}{2c_1^{1/2}a_2^2} \right)a_2\\
&<\frac{1+\frac{1}{c_1A_{1,2}A_3}}{\frac{\chi}{2}-\frac{1}{2c_1^{1/2}A_{1,2}}}\left(0.5c_1^{1/2}+\frac{1}{2c_1^{1/2}A_{1,2}^2} \right)(\rho+3)\log \alpha_2\\
&=:G(B_2^{(1)})\cdot \log \alpha_2.
\end{align*}
On the other hand, if $\max\{S_1,T_1\}=T_1$, then
\begin{align*}
B_2&=\frac{(S_1+1)(T_1+1)}{\mathcal{M}-\max\{S_1,T_1\}}<\frac{(c_1a_1a_3+1)(c_1a_1a_2+1)}{\chi c_1^{3/2}a_1a_2a_3-c_1a_1a_2}\\
&=\frac{1+\frac{1}{c_1a_1a_2}}{\frac{\chi}{2}-\frac{1}{2c_1^{1/2}a_3}}\left(0.5c_1^{1/2}+\frac{1}{2c_1^{1/2}a_2a_3} \right)a_2\\
&<\frac{1+\frac{1}{c_1A_{1,2}^2}}{\frac{\chi}{2}-\frac{1}{2c_1^{1/2}A_{3}}}\left(0.5c_1^{1/2}+\frac{1}{2c_1^{1/2}A_{1,2}A_3} \right)(\rho+3)\log \alpha_2\\
&=:G(B_2^{(2)})\cdot \log \alpha_2,
\end{align*}
where we gave these expressions in the form where it is clear that they are decreasing in variables $a_1$, $a_2$ and $a_3$, so we can use lower bounds of these variables to get an upper bound on $B_2$. Observe that
$$G(B_2^{(1)})=\frac{(c_1A_{1,2}A_3+1)(c_1A_{1,2}^2+1)}{\chi c_1^{3/2}A_{1,2}^2 A_3-c_1A_{1,2}A_3},\quad G(B_2^{(2)})=\frac{(c_1A_{1,2}A_3+1)(c_1A_{1,2}^2+1)}{\chi c_1^{3/2}A_{1,2}^2 A_3-c_1A_{1,2}^2},$$
and since these expressions only differ in their denominators, it is easy to see that if $A_3>A_{1,2}$, then $G(B_2^{(1)})>G(B_2^{(2)})$. Inequality $A_3>A_{1,2}$ will hold for $\rho \in [5.5,14]$, which is a reason why we have chosen that interval for our observations.

Now we define $G(B_2)=\max\{G(B_2^{(1)}),G(B_2^{(2)})\}$, so
$$B_2<G(B_2)\cdot \log \alpha_2.$$

Similarly, we will first assume that $\max\{R_1,T_1\}=R_1$, so
\begin{align*}
B_3&:=\frac{(R_1+1)(T_1+1)}{\mathcal{M}-\max\{R_1,T_1\}}<\frac{(c_1a_2a_3+1)(c_1a_1a_2+1)}{\chi c_1^{3/2}a_1a_2a_3-c_1a_2a_3}\\
&=\frac{1+\frac{1}{c_1a_2a_3}}{\frac{\chi}{2}-\frac{1}{2c_1^{1/2}a_1}}\left(0.5c_1^{1/2}+\frac{1}{2c_1^{1/2}a_1a_2} \right)a_2\\
&<\frac{1+\frac{1}{c_1A_{1,2}A_3}}{\frac{\chi}{2}-\frac{1}{2c_1^{1/2}A_{1,2}}}\left(0.5c_1^{1/2}+\frac{1}{2c_1^{1/2}A_{1,2}^2} \right)(\rho+3)\log \alpha_2\\
&=:G(B_3^{(1)})\cdot \log \alpha_2,
\end{align*}
and if $\max\{R_1,T_1\}=T_1$, then
\begin{align*}
B_3&=\frac{(R_1+1)(T_1+1)}{\mathcal{M}-\max\{R_1,T_1\}}<\frac{(c_1a_2a_3+1)(c_1a_1a_2+1)}{\chi c_1^{3/2}a_1a_2a_3-c_1a_1a_2}\\
&=\frac{1+\frac{1}{c_1a_1a_2}}{\frac{\chi}{2}-\frac{1}{2c_1^{1/2}a_3}}\left(0.5c_1^{1/2}+\frac{1}{2c_1^{1/2}a_2a_3} \right)a_2\\
&<\frac{1+\frac{1}{c_1A_{1,2}^2}}{\frac{\chi}{2}-\frac{1}{2c_1^{1/2}A_3}}\left(0.5c_1^{1/2}+\frac{1}{2c_1^{1/2}A_{1,2}A_3} \right)(\rho+3)\log \alpha_2\\
&=:G(B_3^{(2)})\cdot \log \alpha_2.
\end{align*}
Analogously,  $G(B_3)=\max\{G(B_3^{(1)}),G(B_3^{(2)})\}$ and
$$B_3<G(B_3)\cdot \log \alpha_2.$$
Notice that since we have chosen the same lower bounds on $a_1$ and $a_2$, we have $G(B_2^{(1)})=G(B_3^{(1)})$ and $G(B_2^{(2)})=G(B_3^{(2)})$, and also $G(B_2)=G(B_3)$.
\par
Now, let us observe the case \textit{(A2)}. Here we have some integers $r_1$, $s_1$, $t_1$ and $t_2$, such that
$$(t_1b_1+r_1b_3)s_1=r_1b_2t_2,\quad \gcd(r_1,t_1)=\gcd(s_1,t_2)=1,$$
and
$$|r_1s_1|\leq \delta B_1,\quad |s_1t_1|\leq \delta B_2, \quad |r_1t_2|\leq \delta B_3, \quad \delta=\gcd(r_1,s_1).$$
We have $r_1=\delta r_1'$ and $s_1=\delta s_1'$. Since $b_1=2h$, $b_2=2j$ and $b_3=1$ we also have
$$s_1't_1\cdot 2h+\delta r_1' s_1' =r_1't_2\cdot 2j,$$
and
$$|\delta r_1's_1|\leq B_1,\quad |s_1't_1|\leq B_2, \quad |r_1't_2|\leq B_3.$$

First, let us observe the case when $t_2=0$. Then $\gcd(s_1,t_2)=s_1=1$ and from $(t_1b_1+r_1b_3)s_1=0$, since $s_1\neq 0$, we get $t_1b_1=-r_1b_3$, i.e.\@ $2ht_1=-r_1$. Since $\gcd(r_1,t_1)=1$, we conclude that $t_1=\mp 1$ and $r_1=\pm 2h$. Also, we see from observations stated before that
$$|r_1s_1|=2h\leq B_1<\frac{\left( c_1A_{1,2}+\frac{1}{A_3}\right)(c_1A_{1,2}A_3+1)}{\chi c_1^{3/2}A_{1,2}^2A_3-c_1A_{1,2}A_3}a_3.$$
Since $\chi=2$ and $A=\min\{a_1,a_3\}>1$, we have that
$$c_1=\max\left\{ (\chi ML)^{2/3},\sqrt{2ML/A}\right\}=(2ML)^{2/3}.$$
If we use minimal and maximal values of our parameters $M$ i $L$, we get $$260<c_1<966.$$
Using these values and lower bounds $A_{1,2}>48.9$, $A_3>8\log 10^5>92.1$ and the fact that $a_3<8(1+\frac{0.08675}{\log 10^5}\cdot 13)\log c$, we get the inequality
$$B_1<979.86\log c.$$
So, we see that the inequality $2h<979.86\log c$ holds. From  Proposition \ref{indeksi_ac} we have that $h>0.666662\sqrt{ac}\geq 0.666662\sqrt{c}$, which implies
$$\sqrt{c}<734.91\log c.$$
Solving this inequality in variable $c$, we get $c<1.9701\cdot 10^8$. We will see that this upper bound is much lower than the upper bound we will get in case $t_2\neq 0$.

Now, let us assume that $t_2\neq 0$. We can multiply the linear form $\Lambda_1$ with factor $r_1't_2\neq 0$, and after rearranging we get a linear form in two logarithms
\begin{equation}\label{lin_forma_mignote_A2}
r_1't_2\Lambda_1=2h\log \left( \alpha_1^{r_1't_2}\cdot \alpha_2^{-s_1't_1}\right)-\log\left(\alpha_2^{\delta r_1' s_1'}\cdot \alpha_3 ^{-r_1' t_2} \right),
\end{equation}
where $\delta=\gcd (r_1,s_1)$, $r_1'=\frac{r_1}{\delta}$ and $s_1'=\frac{s_1}{\delta}$. On this form we would like to use the next result from \cite{laurent}.
\begin{theorem}[Laurent]\label{laurent}
Let $a_1'$, $a_2'$, $h'$, $\varrho$ and $\mu$ be real numbers with $\varrho>1$ and $1/3\leq \mu \leq 1$. Set
\begin{align*}
\sigma&=\frac{1+2\mu-\mu^2}{2},\quad \lambda'=\sigma \log \varrho,\quad H=\frac{h'}{\lambda '}+\frac{1}{\sigma},\\
\omega&=2\left( 1+\sqrt{1+\frac{1}{4H^2}}\right),\quad \theta=\sqrt{1+\frac{1}{4H^2}}+\frac{1}{2H}.
\end{align*}
Consider the linear form
$$\Lambda=b_2 \log \gamma_2-b_1 \log \gamma_1,$$
where $b_1$ and $b_2$  are positive integers. Suppose that $\gamma_1$ are $\gamma_2$ multiplicatively independent. Put $D=[\mathbb{Q}(\gamma_1,\gamma_2):\mathbb{Q}]/[\mathbb{R}(\gamma_1,\gamma_2):\mathbb{R}]$, and assume that
\begin{align*}
&h' \geq \max \left\{ D\left(\log\left( \frac{b_1}{a_2'}+\frac{b_2}{a_1'}\right)+\log \lambda'+1.75\right)+0.06,\lambda',\frac{D\log 2}{2}\right\},\\
& a_i'\geq \max\{1,\varrho|\log \gamma_i|-\log|\gamma_i|+2Dh(\gamma_i)\}, \quad i=1,2,\\
&a_1'a_2'\geq \lambda'^2.
\end{align*}
Then
$$\log|\Lambda|\geq -C\left(h'+\frac{\lambda'}{\sigma} \right)^2a_1'a_2'-\sqrt{\omega \theta}\left(h'+\frac{\lambda'}{\sigma} \right)-\log\left(C' \left(h'+\frac{\lambda'}{\sigma} \right)^2a_1'a_2' \right)$$
with
\begin{align*}
&C=\frac{\mu}{\lambda'^3 \sigma}\left(\frac{\omega}{6}+\frac{1}{2}\sqrt{\frac{\omega^2}{9}+\frac{8\lambda'\omega^{5/4}\theta^{1/4}}{3\sqrt{a_1'a_2'}H^{1/2}}+\frac{4}{3}\left( \frac{1}{a_1'}+\frac{1}{a_2'}\right)\frac{\lambda'\omega}{H}   } \right)^2   ,\\
&C'=\sqrt{\frac{C\sigma \omega \theta}{\lambda'^3 \mu}}.
\end{align*}
\end{theorem}
To apply Theorem \ref{laurent}\@ on the linear form (\ref{lin_forma_mignote_A2}) we must first check that the conditions of the theorem are satisfied. Since $\alpha_1$, $\alpha_2$ and $\alpha_3$ are multiplicatively independent, so are $\gamma_1$ and $\gamma_2$.
\\
Now we can assume here that $h\geq G(h)\cdot \log \alpha_2 \log c$ and aim to find the best possible result in this case. If the result we get is better than $h<G(h)\cdot \log \alpha_2 \log c$, we will take $G(h)\cdot \log \alpha_2 \log c$ as an upper bound for $h$.
\par Notice that,
\begin{align*}
h(\gamma_1)&\leq 0.5B_1\log \alpha_2+B_2\log c\\
&<(0.5G(B_1)+G(B_2))\log \alpha_2 \log c=: G(h(\gamma_1))\cdot \log \alpha_2 \log c,\\
h(\gamma_2)&\leq 0.5 B_2 \log \alpha_1+0.5 B_3 \log \alpha_2\\
&<B_3 \log \alpha_2< G(B_3)\cdot \log^2 \alpha_2=:G(h(\gamma_2))\cdot \log^2 \alpha_2 ,\\
|\log \gamma_1|&\leq B_1 \log \alpha_2+0.694 B_3\\
&\leq \left( G(B_1)+0.694\frac{G(B_3)}{\log 10^5} \right)\log \alpha_2 \log c=:G(|\log \gamma_1|)\cdot \log \alpha_2 \log c
\end{align*}
and
\begin{align*}
|\log \gamma_2|&<\frac{B_2+|\log \gamma_1|}{2h}<\frac{G(B_2)\log \alpha_2 +
G(|\log \gamma_1|)\cdot \log \alpha_2 \log c}{2h}\\
&<\frac{\left(\frac{G(B_2)}{\log 10^5}+
G(|\log \gamma_1|)\right) \log \alpha_2 \log c}{2G(h)\cdot \log \alpha_2 \log c}=\frac{\frac{G(B_2)}{\log 10^5}+
G(|\log \gamma_1|)}{2G(h)}=:G(|\log \gamma_2|).
\end{align*}
Now we would like to find which condition must parameters $\varrho$ and $\mu$ satisfy in order to apply Theorem \ref{laurent}\@ and to get the lowest possible upper bound on $h$. First we must choose $a_i'$, $i=1,2$, such that
$$a_i'\geq |\log \gamma_i|(\varrho+1)+8h(\gamma_i),\quad i=1,2.$$
We see that we can set
$$a_1'=(G(|\log \gamma_1|)(\varrho+1)+8G(h(\gamma_1)))\log \alpha_2 \log c=:G(a_1')\log  \alpha_2 \log c,$$
and
$$a_2'=\left(\frac{G(|\log \gamma_2|)}{\log^2 10^{5/2}}(\varrho+1)+8 G(h(\gamma_2))\right)\log^2 \alpha_2=:G(a_2')\log^2 \alpha_2.$$
We have
$$\frac{b_1}{a_2'}+\frac{b_2}{a_1'}\leq \frac{\frac{2}{G(a_2')}+\frac{2h}{G(a_1')}}{\log \alpha_2 \log c}\leq \frac{h\left(\frac{2}{210.81\cdot G(a_2')}+\frac{2}{G(a_1')}\right)}{\log \alpha_2 \log c}$$
where we used that since $c>10^5$ then $h>0.666662\sqrt{ac}>0.666662\sqrt{10^5}>210.81$.
Denote
$$G(F)=\frac{2}{210.81\cdot G(a_2')}+\frac{2}{G(a_1')}$$
and
$$F:=\frac{G(F)\cdot h}{\log \alpha_2 \log c}.$$
Since we will observe only values $\varrho\leq 100$, and since $\frac{D\log 2}{2}=2\log 2<1.4$ and $\lambda'<\frac{3}{2}\log \varrho<7$ we can take
$$h'=4(\log F+\log \lambda')+7.06.$$
Since we assumed that $h\geq G(h) \log \alpha_2 \log c$, we now have $F>G(F)\cdot G(h)$ which implies
$$H=\frac{h'}{\lambda'}+\frac{1}{\sigma}>\frac{4\log (G(F)\cdot G(h))}{\lambda'}+\frac{1}{\sigma}.$$
Using this, for specific values of the parameters $\varrho$ and $\mu$ we can calculate $\omega$, $\theta$, $C$ and $C'$ and by Theorem \ref{laurent}\@ we have
\begin{multline*}
\log |r_1't_2\Lambda_1|>\\
-C\left(h'+\frac{\lambda'}{\sigma} \right)^2a_1'a_2'-\sqrt{\omega \theta}\left(h'+\frac{\lambda'}{\sigma} \right)-\log\left(C' \left(h'+\frac{\lambda'}{\sigma} \right)^2a_1'a_2' \right).
\end{multline*}
Assume that $C'\leq 3C$ (which will be true in all our cases). It holds $\log 3x<10^{-3}x$ for $x\geq 10343$, and in all our cases we will have  $a_1'a_2'>10343$ and also $\sqrt{\omega \theta}<3$. Since we also have $\left(h'+\frac{\lambda'}{\sigma}\right)>1$, we can observe the inequality
$$\log |r_1't_2\Lambda_1|>-C\cdot G(a_1')\cdot G(a_2')(1.001+3\cdot 10^{-4}C^{-1})\left(h'+\frac{\lambda'}{\sigma} \right)^2(\log \alpha_2)^3\log c.$$

We wish to find a minimal positive real number $k$ for which the inequality $\log  \alpha_2<k\cdot \log \alpha_1$ holds. If we use that $\sqrt{ac}<\alpha_2$ and $\alpha_1<\sqrt{ab+4}$ we get $ac<(ab+4)^k$. From Proposition \ref{granicagornja_na_c} we have $ac<237.952b^3$, and since $b>10^5$ we find that inequality holds for $k=3.4753$.

Now, we see that we also have
$$\log |r_1't_2\Lambda_1|<\log B_3-4j\log \alpha_2<\log B_3-4h \log \alpha_1,$$
and since $\log \alpha_2<3.4753\log \alpha_1$, so
\begin{align*}
h&<\frac{3.4753}{4}\left(C\cdot G(a_1')\cdot G(a_2')(1.001+3\cdot 10^{-4}C^{-1})+\frac{\log B_3}{\log 10^5 (\log 10^{5/2})^3}\right)\cdot \\
&\cdot \left(h'+\frac{\lambda'}{\sigma} \right)^2\log^2  \alpha_2 \log c\\
&=:G(h2)\left(h'+\frac{\lambda'}{\sigma} \right)^2 \log^2 \alpha_2\log c.
\end{align*}
Multiplying this expression with $\frac{G(F)}{\log \alpha_2 \log c}$ yields
$$F<G(h2)\cdot G(F)\left(4\log F+4\log \lambda'+7.06+\frac{\lambda'}{\sigma} \right)^2\log \alpha_2$$
and if we insert $\log \alpha_2<\log \sqrt{ac+4}$ and an upper bound for $ac$ we will get an upper bound for $F$, denote it with $F_1$, i.e.\@ $F<F_1$. Now from the definition of $F$ we have
$$h<\frac{F_1}{G(F)}\log \alpha_2 \log c,$$
which gives us an upper bound on $h$ and our goal is to minimize a numerical value $\frac{F_1}{G(F)}$.
\par
As in \cite{petorke}, it is not difficult to see that in the case \textit{(A1)} one obtains smaller values than in the case \textit{(A2)} and therefore smaller upper bounds, so we see it is not necessary to calculate it. \par
Now, it remained to implement the described algorithm for the inequality (\ref{mignote_lambda_ocjena}) and the case \textit{(A2)}. We observed these values of the parameters, $\chi=2$ fixed, $\rho\in[5.5,14]$ with step $0.5$, $L\in [700,1500]$ with step $1$, $M\in [3,10]$ with step $0.1$ and after calculating the upper bound on $h$ by Theorem \ref{mignote}, we also consider all values  $\varrho\in [40,85]$ with step $1$ and $\mu \in [0.44,0.76]$ with step $0.01$ such that the coefficient $G(h2)$ is the least possible one. \\
In the first turn we used  $ac<1.08915\cdot 10^{34}$ and $h<6.95745\cdot 10^{16}$, and the best value was obtained for the parameters $\rho=11.5$, $ M=4.7$ and $ L=1043$ where we got $h<5.66642\cdot 10^9\log \alpha_2 \log c$, and for $\varrho=59$ and $\mu=0.63$ we got $h<4.85941\cdot 10^{10}\log \alpha_2 \log c$ in the case \textit{(A2)}. From this we have $ac<2.42372\cdot 10^{28}$ and $h<1.03788\cdot 10^{14}$.\\
Now these new upper bounds can be used for the second turn and the best value is obtained for the parameters $\rho=11$, $M=4.6$, $L=901$ where we got $h<4.13857\cdot 10^9 \log \alpha_2 \log c$, and for $\varrho=59$, $\mu=0.63$ we got $h<3.53075\cdot 10^{10}\log \alpha_2 \log c$. From this we obtain $ac<1.22705\cdot 10^{28}$ and $h<7.38475\cdot 10^{13}$.\\
We repeat a process three more times, and finally get that $ac<1.17732\cdot 10^{28}$, $h<3.46289\cdot 10^{10}\log \alpha_2 \log c$ and  $h<7.23357\cdot 10^{13}$. This upper bound will be good enough for final steps of the proof so we state the next proposition.
\begin{proposition}\label{granice_ac_i_h}
Let $\{a,b,c,d,e\}$ be a $D(4)$-quintuple, such that $a<b<c<d<e$. Then $ac<1.17732\cdot 10^{28}$. Also,
$h<3.46289\cdot 10^{10}\log \alpha_2 \log c$ and   $h<7.23357\cdot 10^{13}$.
\end{proposition}

\section{$D(4)$-quintuples with regular triples}
Let $\{a,b,c,d,e\}$ be a $D(4)$-quintuple with $a<b<c<d<e$. We have seen that $d=a+b+c+\frac{1}{2}(abc+rst)$ and
\begin{align*}
ad+4=x^2,&\quad bd+4=y^2,\quad cd+4=z^2,\\
x=\frac{at+rs}{2},&\quad y=\frac{rt+bs}{2},\quad z=\frac{cr+st}{2}.
\end{align*}
If $\{a,b,c\}$ is a regular triple, i.e.\@ $c=a+b+2r$, then we also have $s=a+r$, $t=b+r$ and $d=rst$ and by simple calculation we can see that
$$ x=rs-2, \quad y=rt-2, \quad z=st+2. $$
These relations will be helpful in proving some special claims  about $D(4)$-quintuples with $c=a+b+2r.$

\begin{lemma}\label{reg_trojka_n_vece_rpola}
If $\{a,b,c,d,e\}$ is a $D(4)$-quintuple, $a<b<c<d<e$, such that $c=a+b+2r$, then $2n>r$.
\end{lemma}
\begin{proof}
From Lemma \ref{kongruencije}\@ we have
$$a\varepsilon l^2+xl \equiv c\varepsilon n^2+zn (\bmod \ d).$$
Assume that equality holds, i.e. $a\varepsilon l^2+xl = c\varepsilon n^2+zn$. Multiplying by $\varepsilon (zn+xl)$ and rearranging yields
$$(al^2-cn^2)(\epsilon(zn+xl)+d)=4(n^2-l^2).$$
By Lemma \ref{odnosim_n_l}\@ we have $n<l\leq 2n$ (it is easy to check that equality cannot hold), so we have $n\neq l$ and $\frac{1}{2}\leq \frac{n}{l}<1$, which implies $al^2-cn^2|4(n^2-l^2)$, i.e. $al^2-cn^2<4(l^2-n^2)$, since the second factor in  the previous inequality is greater than $1$. Now we have
$$\left|\frac{a}{c}-\left(\frac{n}{l} \right)^2 \right|\leq \frac{4}{c}\left(1-\left(\frac{n}{l}\right)^2 \right).$$
Since $c=a+b+2r>a+a+2a=4a$, we also have $\frac{a}{c}<\frac{1}{4}\leq \left( \frac{n}{l}\right)^2$, so
$$\frac{1}{4}-\frac{a}{c}<\left(\frac{n}{l} \right)^2-\frac{a}{c} \leq \frac{4}{c}\left(1-\left(\frac{n}{l}\right)^2 \right)\leq \frac{3}{c},$$
i.e.\@ it must be $c<4a+12$. But, by Lemma \ref{b_a_57sqrt} we have $b\geq a+57\sqrt{a}$ which would then imply
$$4a+57\sqrt{a}<a+b+2r<4a+12,$$
and this leads to a contradiction since $a\geq 1$. We can now conclude that our assumption was wrong, equality does not hold, so we have
$$d<|al^2-xl-cn^2+zn|\leq |al^2-cn^2|+|xl-zn|.$$
It can be easily seen that $|al^2-cn^2|<cn^2$ and $|xl-zn|<zn$, so we have that
$d<cn^2+zn.$
Assume that $n\leq \frac{r}{2}$. Since $d=rst=r(a+r)(b+r)$ and $z=st+2=(a+r)(b+r)+2$, we have
$$
r(a+r)(b+r)<(a+b+2r)\frac{r^2}{4}+((a+r)(b+r)+2)\frac{r}{2}\\
$$
and after canceling and rearranging we see that this cannot be true. We can now conclude $n>\frac{r}{2}$.
\end{proof}

The next Lemma can be proved similarly as \cite[Lemma 19]{petorke} so we omit a proof.

\begin{lemma}\label{regularne_kongruencije}
Let  $\{a,b,c,d,e\}$ be a $D(4)$-quintuple such that $a<b<c<d<e$ and $c=a+b+2r$. Then
\begin{align*}
8l\equiv 2(1-(-1)^j)(-\varepsilon c)\ (\bmod \ s)&,\quad 8n\equiv 2(1-(-1)^j)\varepsilon a \ (\bmod  \ s)\\
8m \equiv 2(1-(-1)^k)(-\varepsilon c)\ (\bmod \ t)&,\quad 8n\equiv 2(1-(-1)^k)\varepsilon b\ (\bmod \ t),
\end{align*}
where $\varepsilon=\pm1$.
\end{lemma}

\begin{lemma}\label{reg_3_slucaja}
Let $\{a,b,c,d,e\}$ be a $D(4)$-quintuple such that $a<b<c<d<e$ and $c=a+b+2r$. Then, at least one of the following congruences holds
\begin{enumerate}[label=\roman*)]
\item $8l\equiv 8n \equiv 0 (\bmod \ s )$,
\item $8m\equiv 8n \equiv 0 (\bmod \ t )$,
\item $8n \equiv -4\varepsilon r \left(\bmod \ \frac{st}{\gcd (s,t)}\right)$, and $\gcd(s,t)\in\{1,2,4\}$.
\end{enumerate}
\end{lemma}
\begin{proof}
If $j$ is even, then $1-(-1)^j=0$ implies $8l\equiv 8n\equiv 0 (\bmod \ s)$ and $i)$ holds. \\
If $k$ is even, then $1-(-1)^k=0$ implies $8m\equiv 8n\equiv 0 (\bmod \ t)$ and $ii)$ holds. \\
If both $j$ and $k$ are odd, then
\begin{align*}
8l\equiv 4(-\varepsilon c) (\bmod \ s),& \quad 8n\equiv 4\varepsilon a (\bmod \ s),\\
8m\equiv 4(-\varepsilon c) (\bmod \ t),& \quad 8n\equiv 4\varepsilon b (\bmod \ t).
\end{align*}
From $s=a+r$ and $t=b+r$ we have $a\equiv -r (\bmod \ s)$ and $b\equiv -r (\bmod \ t)$, so
$$8n\equiv- 4\varepsilon r (\bmod \ s), \quad 8n\equiv -4\varepsilon r (\bmod \ t),$$
i.e.\@
$$8n\equiv- 4\varepsilon r \left(\bmod \ \frac{st}{\gcd (s,t)}\right).$$
Since $c=s+t$ we can see that $\gcd(s,t)=\gcd(s,s+t)=\gcd(s,c)$, and from $ac+4=s^2$ we conclude $\gcd(s,c)|4$, which proves the statement of the lemma.
\end{proof}

We would like to use these results to obtain some effective bounds on elements $\{a,b,c\}$ in order to use Baker-Davenport reduction.\par
Set
\begin{align*}
&\beta_1=\frac{x+\sqrt{ad}}{2},\quad \beta_2=\frac{y+\sqrt{bd}}{2},\quad \beta_3=\frac{z+\sqrt{cd}}{2}\\
&\beta_4=\frac{\sqrt{c}(\varepsilon \sqrt{a}+\sqrt{d})}{\sqrt{a}(\varepsilon \sqrt{c}+\sqrt{d})},\quad \beta_5=\frac{\sqrt{c}(\varepsilon \sqrt{b}+\sqrt{d})}{\sqrt{b}(\varepsilon \sqrt{c}+\sqrt{d})},
\end{align*}
and consider the following linear forms in logarithms
\begin{align*}
\Lambda_2&=2l\log \beta_1 -2n \log \beta_3+\log \beta_4,\\
\Lambda_3&=2m\log \beta_2-2n \log \beta_3+\log \beta_5.
\end{align*}

From \cite{fil_xy4} we have the next lemma, and to avoid confusion, we would like to emphasize that $v_m$ and $w_n$ here denote sequences connected to the extension of a triple to a quadruple, as in Section $3$.
\begin{lemma}[Lemma 10 in \cite{fil_xy4}]
Let $\{a,b,c,d\}$ be a $D(4)$-quadruple. If $v_m=w_n$, $m,n\neq 0$, then
\begin{align*}
0<m\left( \frac{s+\sqrt{ac}}{2}\right)-&n\log\left( \frac{t+\sqrt{bc}}{2}\right)+\log\frac{\sqrt{b}(x_0\sqrt{c}+z_0\sqrt{a})}{\sqrt{a}(y_1\sqrt{c}+z_1\sqrt{b})}\\
&<2ac\left( \frac{s+\sqrt{ac}}{2}\right)^{-2m}.
\end{align*}
\end{lemma}

We apply this lemma to $D(4)$-quadruples $\{a,b,d,e\}$ and $\{b,c,d,e\}$ to get upper bounds on $\Lambda_2$ and $\Lambda_3$.
\begin{lemma}\label{regulare_granice_formi}
$0<\Lambda_2<2ad \beta_1^{-4l}$ and $0<\Lambda_3<2bd \beta_2^{-4m}$.
\end{lemma}

Now we will observe each case of  Lemma \ref{reg_3_slucaja}\@ to get upper bounds on some elements of a $D(4)$-quintuple.

\begin{lemma}\label{reg_prvi_slucaj}
If $8l\equiv 8n \equiv 0 (\bmod \ s)$, then $s\leq 201884$.
\end{lemma}
\begin{proof}
It is easy to see that $l\equiv n \equiv 0 (\bmod \ \frac{s}{\gcd(s,8)})$, so $l=\frac{s}{\gcd(s,8)}l_1$ and $n=\frac{s}{\gcd(s,8)}n_1$ for some $l_1,n_1\in \mathbb{N}$. Denote $s'=\frac{s}{\gcd(s,8)}$. We have
$$\Lambda_2=2s'l_1\log \beta_1-2s'n_1\log \beta_3+\log \beta_4=\log \beta_4-2s'\log \frac{\beta_3^{n_1}}{\beta_1^{l_1}}.$$
As in \cite{petorke} we can easily see that $\beta_1$ and $\beta_3$ are invertible in $\mathbb{Q}(\sqrt{ad})$ and $\mathbb{Q}(\sqrt{bd})$, so we can take
$$D=4,\quad b_1=2s',\quad b_2=1,\quad \gamma_1=\frac{\beta_3^{n_1}}{\beta_1^{l_1}},\quad \gamma_2=\beta_4.$$
Conjugates of $\gamma_1$ are
$$\frac{\beta_3^{n_1}}{\beta_1^{l_1}},\quad \frac{\beta_3^{-n_1}}{\beta_1^{l_1}},\quad \frac{\beta_3^{n_1}}{\beta_1^{-l_1}},\quad \frac{\beta_3^{-n_1}}{\beta_1^{-l_1}}$$
and depending on whether $\beta_3^{n_1}>\beta_1^{l_1}$ or $\beta_3^{n_1}<\beta_1^{l_1}$ we have
$$h(\gamma_1)=\frac{1}{4}\left(\left|\log \frac{\beta_3^{n_1}}{\beta_1^{l_1}} \right|+ \left|\log \frac{\beta_3^{n_1}}{\beta_1^{-l_1}} \right|\right)=\frac{n_1}{2}\log \beta_3$$
or
$$h(\gamma_1)=\frac{1}{4}\left(\left|\log \frac{\beta_3^{-n_1}}{\beta_1^{-l_1}} \right|+ \left|\log \frac{\beta_3^{n_1}}{\beta_1^{-l_1}} \right|\right)=\frac{l_1}{2}\log \beta_1.$$
By Lemma \ref{regulare_granice_formi}
$$0<\log \beta_4-2s'\log \frac{\beta_3^{n_1}}{\beta_1^{l_1}}<2ad \beta_1^{-4l},$$
so we have
$$
\left|\log \frac{\beta_3^{n_1}}{\beta_1^{l_1}}\right|<\frac{1}{2s'}(\log \beta_4+2ad\beta_1^{-4l})
<\frac{1}{2s'}\left(\log \beta_4+\frac{2}{ad}\right).
$$
It also holds
$$\beta_4=\sqrt{\frac{c}{a}}\left(1-\varepsilon \frac{\sqrt{c}-\sqrt{a}}{\sqrt{d}+\varepsilon \sqrt{c}}\right)\leq \sqrt{\frac{c}{a}}\left(1+ \frac{\sqrt{c}}{\sqrt{d}-\sqrt{c}}\right)< 2\sqrt{\frac{c}{a}},$$
which implies
$$
\left| \log \frac{\beta_3^{n_1}}{\beta_1^{l_1}}\right|<\frac{\log 2\sqrt{\frac{c}{a}}}{2s'}+\frac{2}{2s'ad}<\frac{\log 2s}{2s'}+\frac{1}{s'ad}=\gcd(s,8)\left( \frac{\log 2s}{2s}+\frac{1}{sad}\right).
$$
We can assume $r>10^4$, otherwise $s=a+r<2r<20000$, so we also have $s>10^4$ and $d=rst>r^3>10^{12}$. Now we see
$$\left| \log \frac{\beta_3^{n_1}}{\beta_1^{l_1}}\right|<\gcd (s,8) \left( \frac{\log (2\cdot 10^4)}{2\cdot 10^4}+\frac{1}{ 10^4 \cdot10^{12}}\right)<5\cdot 10^{-4}\gcd(s,8)<0.004.$$
Also
$$\left|\frac{n_1}{2}\log \beta_3-\frac{l_1}{2}\log \beta_1 \right|<0.002$$
and
$$h(\gamma_1)<\frac{l_1}{2}\log \beta_1+0.002.$$
Absolute values of conjugates of $\gamma_2=\beta_4$ are all greater than $1$ and a minimal polynomial can be calculated analogously as for $\alpha_3$ from the previous section so we have
$$h(\gamma_2)\leq \frac{1}{4}\log \left( a^2(d-c)^2\cdot \frac{c^2}{a^2}\cdot \frac{(d-a)^2}{(d-c)^2}\right)<\frac{1}{2}\log (cd)<\log \beta_3.$$
Now, we can apply Theorem \ref{laurent}\@ for parameters $\varrho=61$ and $\mu=0.7$. We have
$\sigma=0.955$ and $3.92<\lambda'<3.93$ and take
$$a_1':=4l_1\log \beta_1+0.264\geq 8h(\gamma_1)+\varrho |\log \gamma_1|-\log|\gamma_1|.$$
Since $c=a+b+2r<4b$ we have $d>abc>c^2/4$ which implies $\beta_3>\sqrt{cd}>\frac{1}{2}c^{3/2}$.\\
We can choose
\begin{align*}
a_2'&:=28\log((1.264)^3\beta_3)>60\log\left(1.264\cdot \frac{1}{\sqrt[3]{2}}\sqrt{c} \right)+8\log ((1.264)^3 \beta_3)\geq\\
&\geq \varrho |\log \gamma_2|-\log |\gamma_2|+8h(\gamma_2).
\end{align*}
From the assumption $r>10^4$ we have $a_1'>56$ and $a_2'>560$, so we see that our choice of parameters is good and we can apply theorem. \\
Set
$$b':=\frac{2s'}{a_2'}+0.018>\frac{b_1}{a_2'}+\frac{b_2}{a_1'}$$
and similarly as in the previous section
$$h'=4\log b' +12.6.$$
Since $\beta_3=\frac{z+\sqrt{cd}}{2}<z$ and $z=st+2<s^3+2$ we have
$$h'>4\log \left( \frac{s'}{14\log ((1.264)^3(s^3+2))}\right)+12.6.$$
Now for  all $4$ values of $\gcd(s,8)$ we calculate values from the Theorem \ref{laurent} which are shown in the next table.
\begin{center}
\begin{tabular}{|c|c|c|c|c|}
\hline
$\gcd(s,8)$ & $1$ & $2$ & $4$ & $8$\\
\hline
$h'$ & $25.508$& $22.736$ & $19.963$ & $17.191$\\
\hline
$H$ & $7.537$& $6.832$ & $6.126$ & $5.421$\\
\hline
$\omega$ & $4.005$& $4.006$ & $4.007$ & $4.0085$\\
\hline
$\theta$ & $1.07$& $1.076$ & $1.085$ & $1.097$\\
\hline
$C$ & $0.02276$& $0.02284$ & $0.02294$ & $0.02307$\\
\hline
$C'$ & $0.04696$& $0.04722$ & $0.04753$ & $0.04792$\\
\hline
\end{tabular}
\end{center}

Define also $B:=\frac{1}{4}\left( h'+\frac{\lambda'}{\sigma}\right)<\log b'+4.187$ which now yields
\begin{align*}
\log |\Lambda_2|&\geq -C\left(h'+\frac{\lambda'}{\sigma} \right)^2a_1'a_2'-\sqrt{\omega \theta}\left(h'+\frac{\lambda'}{\sigma} \right)-\log\left(C' \left(h'+\frac{\lambda'}{\sigma} \right)^2a_1'a_2' \right)\\
&\geq -C\cdot 16B^2 a_1' a_2' -\sqrt{\omega \theta}\cdot 4B-\log (C'\cdot 16B^2 a_1'a_2')\\
&\geq -0.3692 B^2 a_1'a_2'-8.388B-\log(0.7668B^2a_1'a_2').
\end{align*}
On the other hand, from Lemma \ref{regulare_granice_formi} we have
$$
\log|\Lambda_2|<-4s'l_1\log \beta_1+\log 2ad=-s'(a_1'-0.264)+\log 2ad,
$$
therefore
\begin{align*}
s'(a_1'-0.264)<0.3692 B^2 a_1'a_2'+8.388B+\log(0.7668B^2a_1'a_2')+\log 2ad.
\end{align*}
From $a_1'>56$ we have $a_1'-0.264>0.9952a_1'$, so now we can observe
\begin{align*}
\frac{2s'}{a_2'}<0.74197 B^2+\frac{16.857}{a_1'a_2'}B+\frac{2.01}{a_1'a_2'}\log(0.7668B^2a_1'a_2')+\frac{2.01}{a_1'a_2'}\log 2ad,
\end{align*}
i.e.\@
\begin{align*}
b'<0.74197 B^2+\frac{16.857}{a_1'a_2'}B+\frac{2.01}{a_1'a_2'}\log(0.7668B^2a_1'a_2')+\frac{2.01}{a_1'a_2'}\log 2ad+0.018.
\end{align*}
Each addend on the right hand side of the inequality can be compared to $B^2$ and it leads to the inequality
$$b'<0.742116B^2+0.02133<0.742116(\log b'+4.187)^2+0.0213,$$
and from this we get $b'<48.28$ which implies
$$
s'<24.131a_2'<675.668\log((1.264)^3(s^3+2)).
$$
For each $\gcd(s,8)\in\{1,2,4,8\}$ we get $s\leq S_1$ where $S_1\in \{20610,$ $44324,94814,201884\}$, i.e.\@ $s\leq 201884$.
\end{proof}

Similarly we can prove next lemma.

\begin{lemma}\label{reg_drugi_slucaj}
If $8m\equiv 8n \equiv 0 \ (\bmod \ t)$, then $t\leq 127293$.
\end{lemma}

Now we observe the last case from the Lemma \ref{reg_3_slucaja}.

\begin{lemma}\label{reg_treci_slucaj}
If $8n \equiv -4\varepsilon r \left(\bmod \ \frac{st}{\gcd (s,t)}\right)$, $\gcd(s,t)\in\{1,2,4\}$, then $r<9164950$.
\end{lemma}
\begin{proof}
By Lemma \ref{reg_trojka_n_vece_rpola} we see that $n>r/2$ which implies $8n+4r>8n-4r>0$, and depending of $\varepsilon$, we have $8n\pm 4r\geq \frac{st}{\gcd (s,t)}\geq \frac{st}{4}.$ So, it always holds $n\geq \frac{st-16r}{32}\geq \frac{c(r-8)}{32}$. By Lemmas \ref{odnos_h_im} and \ref{odnosim_n_l} we have $h>2m>2n$, which yields $h>\frac{c(r-8)}{16}$. \par
Moreover, from Proposition \ref{granice_ac_i_h}\@  we have
$$h<3.46289\cdot 10^{10}\log \alpha_2 \log c.$$
Since
$$\alpha_2<\sqrt{ac+4}=\sqrt{\frac{16a}{r-8}\cdot \frac{c(r-8)}{16}+4}<\sqrt{16\frac{c(r-8)}{16}+4}$$
and
$$c=\frac{16}{r-8}\frac{c(r-8)}{16}<\frac{16}{10^{5/2}-8}\frac{c(r-8)}{16}<\frac{16}{308}\frac{c(r-8)}{16},$$
we have
$$\frac{c(r-8)}{16}<3.46289\cdot 10^{10}\log \left( \sqrt{16\frac{c(r-8)}{16}+4}\right) \log \left( \frac{16}{308}\frac{c(r-8)}{16}\right).$$
By direct calculation we get
$$\frac{c(r-8)}{16}<1.57493\cdot 10^{13}$$
and since $r^2-3+2r\geq c>3r$ we have $r<9164950$ and
$$h<3.46289\cdot 10^{10}\log (2r) \log (r^2-3+2r)<1.85682\cdot 10^{13}.$$
\end{proof}

From Lemmas \ref{reg_prvi_slucaj}, \ref{reg_drugi_slucaj}\@ and \ref{reg_treci_slucaj}\@ we see that there are only finitely many triples $\{a,b,c\}$ left to check whether they are contained in a $D(4)$-quintuple. In order to deal with these remaining cases we will use a Baker-Davenport reduction method over a linear form
$$\Lambda_1:=2h\log\frac{r+\sqrt{ab}}{2}-2j\log\frac{s+\sqrt{ac}}{2}+\log{\frac{\sqrt{c}(\sqrt{a}+\sqrt{b})}{\sqrt{b}(\sqrt{a}+\sqrt{c})}}.$$
More explicitly, a modification of the Baker-Davenport reduction method, from \cite{dujpet}, which we will use is stated next.

\begin{lemma}[Dujella, Peth\H{o}]\label{bd}
 Assume that $M$ is a positive integer. Let $p/q$ be the convergent of the continued fraction expansion of a real number $\kappa$ such that $q>6M$ and let
$$\eta= \| \mu q\|-M\cdot \|\kappa q \|,$$
where $\| \cdot \|$ denotes the distance from the nearest integer. If $\eta>0$, then the inequality
$$0<J\kappa -K+\mu <AB^{-J}$$
has no solution in integers $J$ and $K$ with
$$\frac{\log(Aq/\eta)}{\log B}\leq J  \leq  M.$$
\end{lemma}

\begin{rm}
Consider the inequality  $\frac{c(r-8)}{16}<1.57493\cdot 10^{13}$ from the proof of Lemma \ref{reg_treci_slucaj}.\@ For a fixed $a$ we can calculate maximal $r$ by putting $c=a+\frac{r^2-4}{a}+2r$, and for smaller values of $a$ we get a much better bound on $r$ than the one calculated in the lemma. For example, for $a=1$ we have $r\leq 63164$. Of course, we must also consider bounds from Lemmas \ref{reg_prvi_slucaj}\@ and \ref{reg_drugi_slucaj}.
\end{rm}

As we said before, we will apply Lemma \ref{bd}\@ to the linear form in logarithms $\Lambda_1$, so we take $J=2h$,  $M=2\cdot 1.85682\cdot 10^{13}$. It took approximately $29$ hours and $45$ minutes to run the algorithm in Wolfram Mathematica $11.1$ package on the computer with Intel(R) Core(TM) i7-4510U CPU @2.00-3.10 GHz processor and in each case we got $J=2h<5$ which cannot be true since $2h>2\cdot 0.666662\sqrt{ac}>2\cdot 0.666662\cdot 10^{5/2}>421$. This proves our next theorem.

\begin{theorem}\label{teorem_regularne}
A regular $D(4)$-triple $\{a,b,a+b+2r\}$ cannot be extended to a $D(4)$-quintuple.
\end{theorem}

\section{$D(4)$-quintuples with non-regular triples}

It remains to show that a non-regular $D(4)$-triple cannot be extended to a quintuple. In the proof of the next two theorems we follow the methods used in Theorems 8\@ and 9\@ from  \cite{petorke}, but as we also said before, results similar to those from \cite{cff}, which we need in order to prove these Theorems, could not be proven for every $D(4)$-quintuple and here we will show how our results from Section \ref{Rickert_section}\@ can again be used in proving some special results for $D(4)$-quintuples for which $c$ is not the smallest possible, i.e.\@ $c\neq a+b+2r$.

\par
\begin{theorem}
A $D(4)$-triple $\{a,b,c\}$ for which $\deg (a,b,c)=1$ cannot be extended to a $D(4)$-quintuple.
\end{theorem}
\begin{proof}
By Lemma \ref{c_granice} we have $c>\max\{ab,4b\}$, and by Lemma \ref{nas_rez_granice}\@ we also know $b>4a$. Moreover, by the definition of the degree of a triple we know that $\{d_{-1},a,b,c\}$ is a regular quadruple. Also, $\{d_{-1},a,b\}$ is a regular triple, so if $d_{-1}>b$, we have $d_{-1}=a+b+2r$, and if $d_{-1}<b$, it can by easily shown that $d_{-1}=a+b- 2r$. So, we have $d_{-1}=a+b\pm 2r$  and $c=d_{+}(a,d_{-1},b)=r(r\pm a)(b\pm r)$.
\par  For $d_{-1}$ we have
$$d_{-1}\geq a+b-2r\geq a+b-2\sqrt{\frac{b^2}{4}+4}>a-1,$$
i.e.\@ $d_{-1}\geq a$ so $c>abd_{-1}\geq a^2b$.\par
Assume that $4a<b\leq k\cdot a$. Now we wish to apply Lemma \ref{ricket_konacno} for $A=a$, $B=b$ and $C=d$, and to do so we must satisfy conditions of Theorem \ref{teorem2.1} and find the greatest $k$ for which we can do so. \\
Since $c> a^2b $ and $10^5<b<k\cdot a$, we have $a>\frac{10^5}{k}$ and $c>\frac{10^5}{k} ab$. Now we see
$$d>abc>\frac{10^5}{k} abab\geq \frac{10^5}{k} \frac{b^2}{k^2} b^2=\frac{10^5}{k^3}b^4.$$
On the other hand, since $b>4a$ we have $b-a>3a$ and $A'=\max\{4(B-A),4A\}=4(B-A)$. This yields
\begin{align*}
\frac{59.488A'B(B-A)^2}{Ag^4}&=237.952\frac{(b-a)^3b}{ag^4}<237.952\frac{(b-a)^3b}{a}\\
&\leq 237.952\left(\frac{k-1}{k}\right)^3kb^3,
\end{align*}
and we see that it is enough to observe $k$'s such that
$$\frac{10^5}{k^3}b>237.952\left(\frac{k-1}{k}\right)^3k.$$
Since $b>10^5$ we get $k\leq 81$, which means that now we can assume $4a<b\leq 81a$. Observe an extension of  a $D(4)$-triple $\{a,b,d\}$ to a $D(4)$-quadruple. For the index $n$, (which refers to an extension to a quadruple and not a quintuple), we have by Lemma \ref{ricket_konacno} that
\begin{equation*}n<\frac{4\log(8.40335\cdot 10^{13}(A')^{\frac{1}{2}}A^{\frac{1}{2}}B^2Cg^{-1})\log(0.20533A^{\frac{1}{2}}B^{\frac{1}{2}}C(B-A)^ {-1}g)}{\log(BC)\log(0.016858A(A')^{-1}B^{-1}(B-A)^{-2}Cg^4)}.
\end{equation*}
We can use $\frac{3}{4}b<b-a<\frac{80}{81}b$ and $1\leq g=\gcd(a,b) \leq a$ and we observe expressions
\begin{align*}
8.40335\cdot 10^{13}(A')^{\frac{1}{2}}A^{\frac{1}{2}}B^2Cg^{-1}&<8.35132\cdot 10^{13} b^3d,\\
0.20533A^{\frac{1}{2}}B^{\frac{1}{2}}C(B-A)^ {-1}g&<0.03423bd,\\
0.016858A(A')^{-1}B^{-1}(B-A)^{-2}Cg^4&>0.0000544b^{-3}d,
\end{align*}
thus we have
$$n<\frac{4\log(8.35132\cdot 10^{13} b^3d)\log(0.03423bd)}{\log(bd)\log(0.000054b^{-3}d)}.$$
Function on the right hand side of the inequality is decreasing in $d$ for $d>0$, and since $d>\frac{10^5}{81^3}b^4>0.1881676b^4$ we obtain
$$n<\frac{4\log(1.571449\cdot 10^{13} b^7)\log(0.006441b^5)}{\log(0.1881676b^5)\log(0.000010161b)}.$$
Similarly as in Proposition \ref{granicagornja_na_c}, we have
$$n\geq \frac{m}{2}>0.309017\sqrt{ac}>0.309017\sqrt{\frac{b}{81}\frac{10^5}{81}\frac{b}{81}b}>0.134046b^{3/2}.$$
By combining the two inequalities we get $b<98416<10^5$ which cannot be true. This means that our assumption was wrong and we have $b>81a$.\par
Now we have an even better lower bound
$$d_{-1}>a+b-2\sqrt{\frac{b^2}{81}+4}>a+b-\frac{2}{9}\sqrt{b^2+324}>a+b-\frac{2}{9}(b+1)>\frac{7}{9}b$$
so
$$c>abd_{-1}>\frac{7}{9}ab^2$$
and $ac>\frac{7}{9}(ab)^2.$\par
Assume now that $81a<b<18.0793a^{3/2}$. Then we have $a>18.0793^{-2/3}b^{2/3}$. Observe that
$$
\frac{59.488A'B(B-A)^2}{Ag^4}=237.952\frac{(b-a)^3b}{ag^4}<1639.12b^{10/3}.
$$
On the other hand, since $d_{-1}>\frac{7}{9}b>\frac{7}{9}10^5$, we have
$$d>abc>d_{-1}a^2b^2>d_{-1}18.0793^{-4/3}b^{10/3}>1639.129b^{10/3},$$
so we can again use Lemma \ref{ricket_konacno}. Notice that $A'=4(B-A)<4B$ and $1\leq g \leq a <b/81$ so we use
\begin{align*}
8.40335\cdot 10^{13}(A')^{\frac{1}{2}}A^{\frac{1}{2}}B^2Cg^{-1}
&<1.86742\cdot 10^{13} b^3d\\
0.20533A^{\frac{1}{2}}B^{\frac{1}{2}}C(B-A)^ {-1}g&<0.0002852bd\\
0.016858A(A')^{-1}B^{-1}(B-A)^{-2}Cg^4&>0.00061b^{-10/3}d,
\end{align*}
to obtain
$$n<\frac{4\log(1.86742\cdot 10^{13} b^3d)\log(0.0002852bd)}{\log(bd)\log(0.00061b^{-10/3}d)}.$$
Moreover
$$
d>abc>d_{-1}a^2b^2>\frac{7}{9}18.0793^{-4/3}b^{4/3}b^3>0.01639b^{13/3}
$$
and since the function on the right hand side of inequality is decreasing in $d$, for $d>0$, we can insert this lower bound on $d$ and get
$$n<\frac{4\log(3.0608\cdot 10^{11} b^{22/3})\log(4.6743\cdot 10^{-6}b^{16/3})}{\log(0.01639b^{16/3})\log(9.99\cdot 10^{-6}b)}.$$
On the other hand,
\begin{align*}
n&\geq \frac{m}{2}>0.309017\sqrt{ac}>0.272527ab>0.272527\cdot 18.0793^{-2/3}b^{2/3} b\\
&>0.03956b^{5/3}
\end{align*}
which gives us $b\leq 8$ after combining the inequalities. This, of course, leads to a contradiction which means that we must have $b>18.0793a^{3/2}$.\par
From $b>18.0793a^{3/2}$ we have
$$a^{5/2}<\frac{r^2-4}{18.0793}.$$
Since by Proposition \ref{granice_ac_i_h},\@ we have  $ac<1.17732\cdot 10^{28}$, this implies $\frac{7}{9}(ab)^2<1.17732\cdot 10^{28}$ i.e.\@ $ab<1.23033\cdot 10^{14}$, which gives us $r\leq 11091997$ and $a\leq 135873$.\par
With these upper bounds, we again apply Baker-Davenport reduction on a linear form in logarithms $\Lambda_1$, with $J=2h$, $M=2\cdot 7.23357\cdot 10^{13}$. For each $\{a,b\}$ we check two options for $c$, namely $c=r(r\pm a)(b\pm r)$. It took $11$ days and $18$ hours to check all possibilities and in each case we had $J=2h<5$, which again cannot be true. This proves our theorem.
\end{proof}

All the remaining cases are covered in the next theorem which concludes the proof of Theorem \ref{glavni}.

\begin{theorem}
A $D(4)$-triple $\{a,b,c\}$ such that $\deg(a,b,c)\geq 2$ cannot be extended to a $D(4)$-quintuple.
\end{theorem}
\begin{proof}
If $\deg(a,b,c)\geq 2$ we have that $d_{-1}=d_{-}(a,b,c)$ and $d_{-2}=d_{-}(a,b,d_{-1})$ are positive integers. Moreover, here we also have $b>4a$, $b>10^5$ and $c>\max\{ab,4b\}$. \par
Since from Proposition \ref{granicagornja_na_c} we have an upper bound on $c$, we will separate our observation in four subintervals
$$c\in \left< ab, a^{\frac{1}{2}}b^{\frac{3}{2}} \right]\cup \left< a^{\frac{1}{2}}b^{\frac{3}{2}},ab^2\right]\cup \left<ab^2,ab^{\frac{5}{2}}\right]\cup \left<ab^{\frac{5}{2}},\frac{237.952b^3}{a}\right]. $$

\textbf{Case I:}  $c\in \left< ab, a^{\frac{1}{2}}b^{\frac{3}{2}} \right]$. \\
Since $c=d_{+}(a,b,d_{-1})$, we have $c>abd_{-1}$ and $ad_{-1}<(ab)^{1/2}$, i.e.\@ $ab>(ad_{-1})^2$. \\
On the other hand, $ac>(ab)(ad_{-1})>(ad_{-1})^3$, therefore
$$r_{(a,d_{-1})}=\sqrt{ad_{-1}+4}<\sqrt{(1.17732\cdot 10^{28})^{1/3}+4}<47697,$$
and since $d_{-1}\neq 0$, we also have $r_{(a,d_{-1})}\geq 3$. Our goal is for each $r\in[3,47696]$ to find all possible pairs $\{a,d_{-1}\}$. Moreover, since $\{a,d_{-1},b\}$ is a $D(4)$-triple, $b$ is obtained as a solution of generalized Pell equation
$$\mathcal{A}\mathcal{V}^2-\mathcal{B}\mathcal{U}^2=4(\mathcal{A}-\mathcal{B})$$
where $\mathcal{A}\mathcal{B}+4=\mathcal{R}^2$, $\mathcal{A}<\mathcal{B}$ are positive integers. We know that all solutions of these equation are of the form
$$\mathcal{V}\sqrt{\mathcal{A}}+\mathcal{U}\sqrt{\mathcal{B}}=(\mathcal{V}_0\sqrt{\mathcal{A}}+\mathcal{U}_0\sqrt{\mathcal{B}})\left( \frac{\mathcal{R}+\sqrt{\mathcal{AB}}}{2}\right)^q,$$
where $q\geq 0$ is integer  and  $(\mathcal{U}_0,\mathcal{V}_0)$ is a solution which satisfy
$$0\leq \mathcal{U}_0 \leq \sqrt{\frac{\mathcal{A}(\mathcal{B}-\mathcal{A})}{\mathcal{R}-2}},\quad  1\leq |\mathcal{V}_0| \leq \sqrt{\frac{(\mathcal{R}-2)(\mathcal{B}-\mathcal{A})}{\mathcal{A}}}. $$
Solutions can also be expressed as binary recurrence sequences
$$\mathcal{U}_0,\quad \mathcal{U}_1=\frac{\mathcal{U}_0\mathcal{R}+\mathcal{V}_0\mathcal{A}}{2},\quad \mathcal{U}_{m+2}=\mathcal{R}\mathcal{U}_{m+1}-\mathcal{U}_m.$$
Then we see that $b=\frac{\mathcal{U}^2-4}{\mathcal{A}}=\frac{\mathcal{V}^2-4}{\mathcal{B}}$, so it also must be true that $\mathcal{A}$ divides $\mathcal{U}^2-4$. \\
Since  $a^2b<ac<1.17732\cdot 10^{28}$ we have $b<\frac{1.17732\cdot 10^{28}}{a^2}\leq \frac{1.17732\cdot 10^{28}}{\mathcal{A}^2},$
so
$$\mathcal{U}<\sqrt{\frac{1.17732\cdot 10^{28}}{\mathcal{A}}+4},\quad |\mathcal{V}|<\sqrt{\mathcal{B}\frac{1.17732\cdot 10^{28}}{\mathcal{A}^2}+4}.$$
Now we observe an algorithm in which for each $\mathcal{R}=r_{(a,d_{-1})}\in [3, 47696]$ we search for divisors $d'$ of $\mathcal{R}^2-4$ such that $1\leq d'\leq \mathcal{R}$ and we set $\mathcal{A}=d'$ and $\mathcal{B}=\frac{\mathcal{R}^2-4}{\mathcal{A}}$. For a fixed pair $(\mathcal{A},\mathcal{B})$ we find all possible solutions $(\mathcal{U}_0,\mathcal{V}_0)$ within given bounds and for each pair we find sequence $\mathcal{U}_m$ up until the upper bound for $\mathcal{U}$ expressed before. For each $\mathcal{U}$ we check if $\mathcal{A}|\mathcal{U}^2-4$ and then take $b=\frac{\mathcal{U}^2-4}{\mathcal{A}}$ and for each possibility $(a,d_{-1})\in \{(\mathcal{A},\mathcal{B}),(\mathcal{B},\mathcal{A})\}$ we can calculate $c=d_{+}(a,b,d_{-1})$ and if $c\in \left< ab, a^{\frac{1}{2}}b^{\frac{3}{2}} \right]$ we can do Baker-Davenport reduction for the triple $\{a,b,c\}$ with parameters as in Theorem \ref{teorem2.1}. It took $7$ hours and $54$ minutes to check all possibilities and we got $J<5$ in each case.
\par
\textbf{Case II:}  $c\in \left< a^{\frac{1}{2}}b^{\frac{3}{2}}, ab^2\right]$. \\
We have $abd_{-1}<c<ab^2$, thus $d_{-1}<\frac{c}{ab}<b$, i.e.\@ $b=\max\{a,b,d_{-1}\}$. By Lemma \ref{granice_za_d_plus}\@ we have
$$a^{\frac{1}{2}}b^{\frac{3}{2}}<c<ad_{-1}b+4b=b(ad_{-1}+4),$$
which yields
$$(ab)^{1/2}<ad_{-1}+4.$$
Similarly, $d_{-2}=d_{-}(a,b,d_{-1})$, therefore $b>ad_{-1}d_{-2}$ and $d_{-2}<\frac{b}{ad_{-1}}<\frac{b}{(ab)^{1/2}-4}$. Now we have
$$ad_{-2}<(ab)^{1/2}\frac{(ab)^{1/2}}{(ab)^{1/2}-4}=(ab)^{1/2}\left(1+\frac{4}{(ab)^{1/2}-4} \right)<1.01282(ab)^{1/2},$$
and also we can see that $ab>\left(\frac{ad_{-2}}{1.01282} \right)^{2}$. Moreover
$$ad_{-2}<1.01282(ad_{-1}+4)=1.01282ad_{-1}+4.05128,$$
so $$\frac{ad_{-2}-4.05128}{1.01282}<ad_{-1}.$$
Now,
$$ac>(ab)(ad_{-1})>\left(\frac{ad_{-2}}{1.01282} \right)^{2}\frac{ad_{-2}-4.05128}{1.01282}$$
and since $ac<1.17732\cdot 10^{28}$, we get $ad_{-2}<2.30408\cdot 10^9$ and
$$r_{(a,d_{-2})}=\sqrt{ad_{-2}+4}<48001. $$
We also know that $d_{-1}<b<c^{2/3}<(1.17732\cdot 10^{28})^{2/3}<35.17524\cdot 10^{18}.$ Similarly as in the first case, algorithm is done for $\mathcal{R}=r_{(a,d_{-2})}$ where we search for pairs $(\mathcal{A,B})$, but we set $d_{-1}=\frac{\mathcal{U}^2-4}{\mathcal{A}}$ and observe both possibilities $(a,d_{-2})\in \{(\mathcal{A},\mathcal{B}),(\mathcal{B},\mathcal{A})\}$ and $b=d_{+}(a,d_{-1},d_{-2})$, $c=d_{+}(a,b,d_{-1})$. It took $1$ hours and $34$ minutes to do the reduction and we got $J<5$ in each case.
\par
\textbf{Case III:}  $c\in \left< ab^2, a^{\frac{3}{2}}b^{\frac{5}{2}}\right]$. \\
Here we have $(ab)^2<ac<1.17732\cdot 10^{28}$, so $r=\sqrt{ab+4}\leq 10416543$.
It can be shown that $b<d_{-1}<\frac{c}{ab}$, therefore we have $d_{-1}<\frac{a^{\frac{3}{2}}b^{\frac{5}{2}}}{ab}=a^{1/2}b^{3/2}.$ Since $b<d_{-1}$, we have $d_{-1}=d_{+}(a,b,d_{-2})$ and $d_{-1}>abd_{-2}$, i.e.\@
$$ad_{-2}<\frac{d_{-1}}{b}<(ab)^{1/2}<r\leq 10416543$$
and $r_{(a,d_{-2})}=\sqrt{ad_{-2}+4}<3228.$\\
The algorithm is  similar as in Case $II.$, except $b$ and $d_{-1}$ exchange definition, so $b=\frac{\mathcal{U}^2-1}{\mathcal{A}}$, and $d_{-1}=d_{+}(a,b,d_{-2})$. It took less than $3$ minutes to check all possibilities and we got $J<5$ in each case.

\par \textbf{Case IV:}  $c\in \left<a^{\frac{3}{2}}b^{\frac{5}{2}}, \frac{237.952b^3}{a}\right]$. \\
Here we have $a^{\frac{3}{2}}b^{\frac{5}{2}}< \frac{237.952b^3}{a}$, which yields  $b>\frac{a^5}{237.952^2}$ and
$$1.17732\cdot 10^{28}>ac>(ab)^{5/2}>(a^6\cdot 237.952^{-2})^{5/2},$$
therefore we get $a\leq 460$.\\
As in Case $III.$, here we have $b<d_{-1}$, $d_{-1}=d_{+}(a,b,d_{-2})$ and $c=d_{+}(a,b,d_{-1})$. Therefore
$$d_{-1}<\frac{c}{ab}<\frac{237.952b^2}{a^2},\quad  d_{-2}<\frac{d_{-1}}{ab}<\frac{237.952b}{a^3}.$$
From $(ab)^{5/2}<1.17732\cdot 10^{28}$ we have $ab<1.69184 \cdot10^{11}$. Also, from $a^4d_{-2}<237.952ab$ we get $ad_{-2}<\frac{4.02576\cdot 10^{13}}{a^3}$, thus $r_{(a,d_{-2})}<\frac{6344883}{a^{3/2}}$.\\
Since $a\leq 460$, it is more efficient if we, for each fixed $a$, search $r_{(a,d_{-2})}$ inside interval $\left[3,\frac{6344883}{a^{3/2}}\right]$ such that $a|r_{(a,d_{-2})}^2-4$ and set $d_{-2}=\frac{r_{(a,d_{-2})}^2-4}{a}$ and do similarly as in previous cases. It took $9$ days and $21$ hour to check all possibilities and again we got $J<5$ in each case.
\end{proof}

\bigskip
\textbf{Acknowledgement:} The authors are supported by Croatian Science Foundation under the project no. 6422. The authors are also very grateful to Mihai Cipu who made many important
remarks and suggestions on previous version of this paper. \\

\bigskip
Marija Bliznac Trebje\v{s}anin\\
Faculty of Science, University of Split,\\ Ru\dj{}era Bo\v{s}kovi\'{c}a 33, 21000 Split, Croatia \\
Email: marbli@pmfst.hr \\
\\
Alan Filipin\\
Faculty of Civil Engineering, University of Zagreb,\\ Fra Andrije Ka\v{c}i\'{c}a-Mio\v{s}i\'{c}a 26, 10000 Zagreb, Croatia \\
Email: filipin@grad.hr \\

\end{document}